\pdfoutput=1
\documentclass[11pt,letterpaper]{amsart} %a4paper vs letterpaper

\usepackage{comment}
\usepackage{ifluatex}
\usepackage[english]{babel} % English Language settings

\usepackage{amscd,amsmath,amssymb,amsthm} % AMS Math
\usepackage{multirow}                     % For Multiple Rows in Tables

\ifluatex
  \usepackage{fontspec}
\else
  \usepackage[utf8]{inputenc} % Allows to use UTF8 encoded symbols directly (e.g.  ä ö ü ß é)
  \usepackage{lmodern}
  \usepackage[T1]{fontenc}
\fi

\usepackage[margin=1in]{geometry}    %%%% Changes for Duke format on 16.04.2025

\usepackage{mathtools}
\usepackage{bm} %bold symbols
\usepackage{dutchcal} %% lower case caligraphic  %%16.05.

\usepackage{slashed} % Slashed symbols e.g. to typeset Dirac operators
\usepackage[babel=true]{microtype}
\usepackage[autostyle=true]{csquotes} % Correctly typeset via quotes via \enquote
\usepackage[dvipsnames]{xcolor}
\usepackage[biblatex=true]{embrac}

\usepackage[%
bookmarks=true,
colorlinks,
linkcolor=blue,
urlcolor=blue,
citecolor=blue,
plainpages=false,
pdfpagelabels,
final,
breaklinks=true,
pdfusetitle,
]{hyperref}
\usepackage[citestyle=numeric-comp,
            bibstyle=numeric-comp,
            backend=biber,
            url=false,
            doi=true,
            isbn=false,
            giveninits=true,
            maxbibnames=100,
            sorting=nyt]{biblatex}

\AtEveryBibitem{\clearfield{month}}
\AtEveryBibitem{\clearfield{day}}
\AtEveryBibitem{%
  \ifentrytype{thesis}
    {}
    {
      \ifentrytype{online}{}
      {
      \clearfield{url}%
      \clearfield{urldate}%
      }
    }%
}
            
\usepackage{tikz}
\usetikzlibrary{cd} % Commutative diagrams

\usepackage[all]{xy}

\usepackage[shortlabels]{enumitem}
\newlist{myenumi}{enumerate}{1}
\setlist[myenumi,1]{label=\upshape(\roman*)}
\newlist{myenuma}{enumerate}{1}
\setlist[myenuma,1]{label=\upshape(\alph*)}
\usepackage{color}
\usepackage[textwidth=0.8in]{todonotes}
\usepackage{thmtools}
\declaretheorem[name=Theorem, numberwithin=section]{theorem}
\declaretheorem[name=Theorem, numbered=no]{theorem*}
\declaretheorem[name=Lemma,numberlike=theorem]{lemma}
\declaretheorem[name=Lemma,numbered=no]{lemma*}
\declaretheorem[name=Corollary,numberlike=theorem]{cor}
\declaretheorem[name=Proposition,numberlike=theorem]{prop}
\declaretheorem[name=Definition,numberlike=theorem, style=definition]{definition}

\declaretheorem[name=Example, numberlike=theorem, style=remark]{example}
\declaretheorem[name=Remark, numberlike=theorem, style=remark]{rem}

\declaretheorem[name=Theorem]{thmx}
\declaretheorem[name=Corollary, numberlike=thmx]{corx}

\numberwithin{equation}{section}
\allowdisplaybreaks[1]
\usepackage[noabbrev]{cleveref} %Important to load after numberwithin!
\crefname{figure}{Figure}{Figures}
\crefname{table}{Table}{Tables}
\crefname{theorem}{Theorem}{Theorems}
\crefname{thmx}{Theorem}{Theorems}
\crefname{lemma}{Lemma}{Lemmas}
\crefname{definition}{Definition}{Definitions}
\crefname{setup}{Setup}{Setups}
\crefname{conjecture}{Conjecture}{Conjectures}
\crefname{question}{Question}{Questions}
\crefname{cor}{Corollary}{Corollaries}
\crefname{corx}{Corollary}{Corollaries}
\crefname{prop}{Proposition}{Propositions}
\crefname{example}{Example}{Examples}
\crefname{rem}{Remark}{Remarks}
\crefname{section}{Section}{Sections}
\crefname{subsection}{Subsection}{Subsections}
\crefname{chapter}{Chapter}{Chapters}
\crefname{appendix}{Appendix}{Appendices}
\crefdefaultlabelformat{#2\textup{#1}#3}
\creflabelformat{enumi}{(#2#1#3)}

\usepackage{crossreftools}
\usepackage{xparse}

\usepackage{ourmacros}

\newcommand{\Hol}{{\operatorname{Hol}}}
\newcommand{\vertical}{\mathrm{vert}}
\newcommand{\incl}{\mathrm{incl}}

\title{Fibrewise Orbifold Resolutions with Applications to $\GTwo$-Moduli Spaces}

\hypersetup{
  pdfauthor={Thorsten Hertl}
}

\author{Thorsten Hertl}
\address[T.~Hertl]{School of Mathematics and Statistics, The University of Melbourne, Australia }
\email{\href{mailto:thorsten.hertl@unimelb.edu.au}{thorsten.hertl@unimelb.edu.au}}
\urladdr{\href{https://thorsten-hertl.github.io/}{https://thorsten-hertl.github.io/}}

\date{\today}

\begin{document}

\begin{abstract}
  By resolving the singularities of tailor-made orbifolds via twisted families of blow-ups, we construct manifold bundles $M \rightarrow E \rightarrow S^2$.
  Using tools from real homotopy theory, we show that these bundles determine a free subgroup in $\pi_2(B\mathrm{hAut}(M)_0)$.
  The proof relies on a generalisation of Sullivan's result, which describes the real homotopy groups of the monoid of homotopy automorphisms $\mathrm{hAut}(X)$ in terms of derivations of the minimal model of $X$, to the monoid $\mathrm{hAut}_A(X)$ of relative homotopy automorphisms.

  As an application, we prove that the moduli space of torsion-free $\mathrm{G}_2$-structures arising from many generalised Kummer constructions contains a free subgroup of positive rank in its second homotopy group.
\end{abstract}

\maketitle

%\tableofcontents

\section{Introduction}\label{Section - Introduction}

A $\GTwo$-structure is a differential form $\varphi \in \Omega^3(M)$ on a seven-dimensional manifold $M^7$ that is \emph{positive} in the sense that the symmetric bilinear form $g_{\varphi}$ defined implicitly by the formula
\begin{equation*}
    g_{\varphi}(v,w) \cdot \vol_{g_\varphi} = \iota_v(\varphi) \wedge \iota_w(\varphi) \wedge \varphi
\end{equation*}
is positive definite.
If $\varphi$ is parallel with respect to the Levi-Civita connection of its underlying metric $g_\varphi$, then we call $\varphi$ \emph{torsion-free} and the pair $(M,\varphi)$ a $\GTwo$-manifold because the holonomy group $\mathrm{Hol}(g_\varphi)$ is then contained in $\GTwo$.
For a closed manifold, full holonomy $\Hol(g_\varphi) = \GTwo$ is attained if and only if, in addition, its fundamental group $\pi_1(M)$ is finite \cite[Proposition 10.2.2]{Joyce2000SpecialHolonomy}.

The group of diffeomorphisms $\Diff(M)$ and its path component of the identity $\Diff(M)_0$ act on the space of all torsion-free $\GTwo$-structures $\tfGTwoStr(M)$ via pull back.
In \cite{Joyce1996CompactG2I}, Joyce not only produced the first example of a closed $\GTwo$-manifold, but also proved that the \emph{moduli space} $\GTwoModuli{M}$ of each closed seven-dimensional manifold is itself a smooth manifold\footnote{If $M$ is not a $\GTwo$-manifold, then the moduli space is empty, of course.} of dimension $b^3(M)$, the third Betti number of $M$, by showing that the period map $\GTwoModuli{M} \rightarrow H^3(M;\R)$, which sends an equivalence class of a torsion-free $\GTwo$-structure to its de Rham cohomology class, is a local diffeomorphism.

In contrast to the moduli spaces of $\KThree$-surfaces, where the analogous period map is an embedding with an explicitly described image, Joyce's result is essentially of local nature and cannot be used to deduce much about the global topological features of $\GTwo$-moduli spaces. 
Thus, understanding global topological properties of $\GTwo$-moduli spaces has become an active field of investigation over the last decade, see \cite{crowley2015newinvariants}, \cite{crowley2025analytic}, \cite{scaduto2020computing}, and \cite{crowley2025PathComponents}; see also \cite{thakar2026g2manifold} for a recent development on the fundamental group of the full quotient $\tfGTwoStr(M)/\Diff(M)$.  

The main purpose of this article is to study $\GTwo$-moduli spaces from the view point of real homotopy theory.
Our first main result shows that $\GTwo$-moduli spaces of generalised Kummer constructions in the sense of Joyce \cite[Chapter 11]{Joyce2000SpecialHolonomy} often have non-trivial real second homotopy groups.
\begin{thmx}\label{Main Thm: Detection G2 Moduli Space}
    Let $M^7$ be a simply connected generalised Kummer construction obtained from $T^7/\Gamma$ by resolving its singularities.
    Let $\mathrm{N}$ be the number of path components $S$ of the singular set of $T^7/\Gamma$ that satisfy the following two properties:
    \begin{itemize}
        \item[(i)] $S$ has a tubular neighbourhood isometrically isomorphic to $T^3 \times D^4/\Z_2$, where $\Z_2$ acts antipodally on $D^4$.
        \item[(ii)] There is a $S'$ different from $S$ satisfying (i) such that their real homology classes $[S], [S'] \in H_3(T^7/\Gamma;\R)$ generate the same vector space.
    \end{itemize}
    Then 
    \begin{equation*}
        \pi_2( \GTwoModuli{M} )\otimes \R \supseteq \R^{\mathrm{N}}.
    \end{equation*}
\end{thmx}
 The conditions in Theorem \ref{Main Thm: Detection G2 Moduli Space} are easy to check in practice.
 We indicate this by discussing in Section \ref{Section: G2 Applications} below some of Joyce's examples in \cite{Joyce1996CompactG2II}.
 The results are summarised in the following table:
\begin{figure}[h]
    \begin{tabular}{ |p{5cm}||p{0.9cm}|p{1cm}|p{1cm}|p{1cm}|p{1cm}|p{1cm}|p{1cm}|p{1cm}|  }
        \hline
      \multicolumn{9}{|c|}{Examples of generalised Kummer constructions from \cite{Joyce1996CompactG2II}} \\
        \hline
 Example number & $3$  & $4$ & $6$ & $7$ & $9$ & $11$ & $13$ & $14$ \\
 \hline
  $\Z^{\mathrm{N}} \subseteq \pi_2(\GTwoModuli{M})$& $\Z^{12}$ & $\Z^8$ & $\Z^4$ & $\Z^2$ & $\Z^{10}$& $\Z^4$ & $\Z^2$ & $\Z^2$ \\
 \hline
\end{tabular}
%\caption{List of Joyce's Examples presented in \cite{Joyce1996CompactG2II} together with the rank $\mathrm{N}$ of the subgroup $\Z^N$ of $\pi_2$ of their moduli spaces provided by Theorem \RefHere}
\label{fig - Table Examples}
\end{figure}

The main challenge in constructing non-trivial maps into the moduli space $\GTwoModuli{M}$ stems from the fact it does not have a universal family.
To remedy this defect, the author and his collaborators introduced, in \cite{crowley2025PathComponents}, the \emph{homotopy moduli space} $\hGTwoModuli{M}$.
This space has the universal property that every family of $\GTwo$-manifolds parametrised by a base space $B$ (up to a suitable notion of equivalence) gives rise to a unique homotopy class of continuous maps $B \rightarrow \hGTwoModuli{M}$.
It further comes with two comparison maps $\hGTwoModuli{M} \rightarrow B\Diff(M)$ and $\hGTwoModuli{M} \rightarrow \GTwoModuli{M}$, the latter inducing an injective homomorphism on all higher homotopy groups, see \cite[Theorem B]{crowley2025PathComponents}.

With the homotopy moduli space at hand, we can sketch a proof of Theorem \ref{Main Thm: Detection G2 Moduli Space}:
The construction carried out in \cite[Section 6]{crowley2025PathComponents} yields, for each singularity component $S$ of the flat orbifold $T^7/\Gamma$ satisfying condition (i), a fibre bundle\footnote{In \cite{crowley2025PathComponents}, this bundle would be denoted by $E_{M,\{S\}}$.} $(E_S,\bm{\varphi}) \rightarrow S^2$ with $\{\varphi_b\}_{b \in S^2} = \bm{\varphi} \in \Omega^3(T^{\vertical,\vee}E_{S})$ a fibre-wise torsion-free $\GTwo$-structure.
Its classifying map into the homotopy moduli space yields a unique homotopy class of continuous maps
\begin{equation}\label{eq: classifying maps of functors}
    \xymatrix{  & \hGTwoModuli{M} \ar[d]& \\ 
    S^2 \ar[r]_-{f_{E_S}} \ar[ru]^-{f_{E_S,\bm{\varphi}}} & B\Diff(M)_0 \ar[r] & B\hAut(M)_0   }
\end{equation}
 and Theorem \ref{Main Thm: Detection G2 Moduli Space} is now a consequence of the following differential topological result.
\begin{thmx}\label{Main Thm: G2 Manifold fibration detection}
    Let $M$ be a simply connected generalised Kummer construction obtained from $T^7/\Gamma$ as in Theorem \ref{Main Thm: Detection G2 Moduli Space} and $\{S\}$ the set of all singularity components satisfying condition (i) and (ii).
    Then the set $\{[f_{E_S}] \in \pi_2(B\hAut(M)_0)\}$ is linearly independent in $\pi_2(B\hAut(M))\otimes \R$.
    In particular, 
    \begin{equation*}
        \mathrm{im}\bigl[  \pi_2(B\Diff(M)_0) \longrightarrow \pi_2(B\hAut(M)_0) \bigr] \otimes \R \supseteq  \R^\mathrm{N} .
    \end{equation*}
\end{thmx}
The comparison map $\hGTwoModuli{M} \rightarrow B\Diff(M)_0$ classifies the functor that sends a $\GTwo$-family $(E,\bm{\varphi})$ to its underlying $M$-fibre bundle $E$, while $B\Diff(M)_0 \rightarrow B\hAut(M)_0$ classifies the forgetful functor that considers the (isomorphism class of the)  $M$ fibre bundle merely as a Serre fibration (up to fibre-homotopy equivalence).
Thus, the geometric interpretation of Theorem \ref{Main Thm: G2 Manifold fibration detection} is that the underlying topological families $E = \{M_b\}_{b \in S^2}$ are pair-wise fibre-homotopy inequivalent.

Theorem \ref{Main Thm: G2 Manifold fibration detection} can be derived from a robustness principle for orbifold resolutions, applicable to a convenient class of smooth orbifolds of dimension at least $4$.
To explain it, we need to generalise the notation from Theorem \ref{Main Thm: Detection G2 Moduli Space}.
Let $X$ be a closed, smooth orbifold in the sense of \cite{Satake1956Orbifolds} and $\mathcal{S} \subseteq X$ be a closed, smooth manifold satisfying the following properties:
\begin{itemize}
    \item[(i)] Each connected component $S \subseteq \mathcal{S}$ has a tubular neighbourhood inside $X$ of the form $\mathrm{Tub}(S) = S \times D^4/\Z_{k_S}$ with $k_S \in \{1,2\}$.
    \item[(ii)] $\mathcal{S}$ contains all singular points of $X$.
\end{itemize}
We call such a pair $(X,\mathcal{S})$ a \emph{tailor-made} orbifold\footnote{We remark that the set $\mathcal{S}$ is allowed to more than just singular points}.
Let $\mathsf{N} \subseteq \pi_0(\mathcal{S})$ be the \emph{nice} subset of path components that have a `homological partner' in the sense that $S \in \mathsf{N}$ if and only if there is a $S' \in \pi_0(\mathcal{S})$ different from $S$ such that the image of their fundamental classes generate the same subvector space $\R \cdot [S] = \R \cdot [S'] \subseteq H_{n-4}(X;\R)$.

By condition (ii), we can resolve the orbifold $X$ to a manifold $M$ by replacing the tubular neighbourhood $\mathrm{Tub}(S) \cong S \times D^4/\Z_{k_S}$ of each singularity component by $S \times D\mathcal{O}(\pm k_S)$, where $D\mathcal{O}(k)$ is the disc-bundle of the complex line bundle $\mathcal{O}(k) \rightarrow S^2$.
Resolving a fixed path component $S \in \pi_0(\mathcal{S})$ in a twisted fashion - that is, using the non-trivial fibre bundle $\mathcal{DO}(k_S) \rightarrow S^2$ defined in Subsection \ref{Subsection: Top Properties Blow Up Families} below instead of $D\mathcal{O}(k_S) \times S^2$ (and the remaining singularities in an untwisted fashion) - gives rise to a fibre bundle $M \rightarrow E_S \rightarrow S^2$, which is classified by a continuous map $f_{E_S} \colon S^2 \rightarrow B\Diff(M)_0$ that is unique up to homotopy.
\begin{thmx}\label{Main thm: Orbifold Robustness}
    Let $(X,\mathcal{S})$ be a tailor-made orbifold, let $\mathsf{N} \subseteq \pi_0(\mathcal{S})$ be the subset of components with a homological partner, and let $M$ be a resolution of $X$ obtained from the construction above.
    Then the set $\{[f_{E_S}] \in \pi_2(B\hAut(M)_0) \, : \, S \in \mathsf{N}\}$ is linearly independent in $\pi_2(B\hAut(M)_0) \otimes \R$.
    In particular,
    \begin{equation*}
        \mathrm{im}\bigl[  \pi_2(B\Diff(M)_0) \longrightarrow \pi_2(B\hAut(M)_0) \bigr] \otimes \R \supseteq  \R^{|\mathrm{N}|} .
    \end{equation*}
\end{thmx}

To demonstrate the effectiveness of Theorem \ref{Main thm: Orbifold Robustness}, we provide two immediate consequences:
\begin{corx}\label{Main Cor: CP conn sum}
    For all $m,n \in \N_0$ with $m + n \geq 2$, we have 
    \begin{equation*}
        \mathrm{im}\bigl[  \pi_2(B\Diff(\CP^{2,\sharp m} \sharp \bar{\CP}^{2,\sharp n})_0) \longrightarrow \pi_2(B\hAut(\CP^{2,\sharp m} \sharp \bar{\CP}^{2,\sharp n})_0) \bigr] \otimes \R \supseteq  \R^{m+n}.
    \end{equation*}
\end{corx}
\begin{proof}
    Let $X = S^4$ and $\mathcal{S} = \{x_1,\dots,x_m,y_1,\dots,y_n\}$ be a set of at least two points.
    Every point here is regular and has a small tubular neighbourhood of the form $D^4$, and these neighbourhoods are pairwise disjoint.
    Resolving the points $\{x_1,\dots,x_m\}$ by $D\mathcal{O}(1)$ and the points $\{y_1,\dots,y_n\}$ by $D\mathcal{O}(-1)$ yields the manifold $\CP^{2,\sharp m} \sharp \bar{\CP}^{2,\sharp n}$.
    Since $S^4$ is connected, $\mathsf{N} = \pi_0(\mathcal{S})$, and the result follows from Theorem \ref{Main thm: Orbifold Robustness}.
\end{proof}
\begin{rem}
    The condition $m+n\geq 2$ cannot be dropped (at least for our construction).
    If we resolve only a single point $x_1 \in S^4$ with the construction from above, we obtain the bundle $E_{\{x_1\}} = \mathcal{DO}(1) \cup_{S^3\times S^2} D^4 \times S^2$.
    But it was shown in \cite{Hertl2026Moduli} that the order of the element $[f_{E_{\{x_1\}}}] \in \pi_2(B\Diff(\CP^2))$ is finite, essentially because one can write down transition functions for $E_{x_1}$ that take values in $\mathrm{PU}(3)$, so $[f_{E_{\{x_1\}}}]$ lives in the image of the homomorphism $\pi_2(B\mathrm{PU}(3)) \rightarrow \pi_2(B\Diff(\CP^2))$ and the domain $\pi_2(B\mathrm{PU}(3)) = \pi_1(\mathrm{PU}(3)) = \Z_3$ has order $3$.
\end{rem}
%
%The next result concerns the diffeomorphism group of the $\KThree$-surface.
%
\begin{corx}
    For the $\KThree$-surface, we have
    \begin{equation*}
        \mathrm{im}\bigl[\pi_2( B\Diff(\KThree)_0 ) \rightarrow \pi_2( B\hAut(\KThree)_0 ) \bigr] \otimes \R \supseteq \R^{16}.
    \end{equation*}
\end{corx}
\begin{proof}
    Recall that the $\KThree$-surface can be obtained from the Kummer construction, see \cite[Example 7.3.2]{Joyce2000SpecialHolonomy}, as follows: The group $\Z_2 \curvearrowright T^4 = \R^4/\Z^4$ acts diagonally by sending $x + \Z^4$ to $-x + \Z^4$.
    This action is not free, but has 16 fixed-points $\mathcal{S} = \{0,1/2\}^{4} = \mathsf{N}$.
    A small tubular neighbourhoods around each point is of the form $D^4/\Z_2$, and we obtain the Kummer surface by resolving these singularities with $D\mathcal{O}(-2)$.
    Theorem \ref{Main thm: Orbifold Robustness} now implies the claim.
\end{proof}
Recently, Baraglia proved in \cite{Baraglia2023K3families} that $\pi_2(B\Diff(\KThree))$ contains an infinitely generated free abelian subgroup.
We believe that our elements form a subgroup of Baraglia's subgroup, which implies that some of these elements are actually in different fibre-homotopy equivalence classes.
In addition, Baraglia's result relies heavily on family Seiberg-Witten invariants, which are not available to higher dimensions at the time of writing.

The reason we study the composition $S^2 \rightarrow \hGTwoModuli{M} \rightarrow B\Diff(M)_0 \rightarrow B\hAut(M)$ apart from the intrinsic interest of distinguishing the $\GTwo$-families as Serre-fibrations, is that the monoid of homotopy automorphisms $\hAut(M)$ can be effectively studied with the help of rational (or real) homotopy theory, as it allows to compute the rational (or real) homotopy groups $\pi_k(\hAut(X))\otimes\bbK$ of a finite nilpotent CW complex $X$ in terms of derivations on its minimal model.
This idea was originally pioneered by Sullivan \cite{Sullivan1977Infinitesimal} and subsequently refined by others; see \cite{Lupton2007RankFundMappingSpaces} \cite{Buijs2008RationalLieAlgebraFunction}, or \cite{berglund2020rational} for a non-exhaustive list.

The next result generalises Sullivan's observation to the relative situation, namely to the topological monoid $\hAut_A(X)$ of all homotopy equivalences on $X$ that restrict to the identity on $A$, which we believe to be of independent interest.
Recall that a continuous map $f \colon X \rightarrow X$ is a homotopy-equivalence relative to $A$ if $f|_A = \id_A$ and if there is a continuous map $g \colon X \rightarrow X$ with $g|_A = \id_A$ such that $g\circ f$ and $f \circ g$ are homotopic to the identity of $X$ through homotopies $H$ with $H_t|_A = \id_A$ for all $t \in [0,1]$.

For a commutative differential graded algebra (cdga) $B_1$ over a field $\bbK$ of characteristic zero and a differential graded ideal $I \subseteq B_2$, let $\mathrm{Der}_n(B_1,I)$ denote the vector space of all differential $B_1 \rightarrow I$ that lower the degree by $n$.
It is easy to see that $\delta \colon \mathrm{Der}_{n}(B_1,I) \rightarrow \mathrm{Der}_{n-1}(B,I)$ given by $\delta(\theta) = d\theta - (-1)^n\theta d$ forms a (homologically graded) differential on $\mathrm{Der}(B_1,I) = \bigoplus \Der_n(B_1,I)$.
\begin{thmx}\label{Main Thm: Homotopy Automorphism via Derivations}
    Let $\iota \colon A \hookrightarrow X$ be a pair of finite, nilpotent CW complexes, $\Lambda V_X$ a Sullivan model for $X$, and $\mathsf{A}(\iota) \colon \mathsf{A}(X) \twoheadrightarrow \mathsf{A}(A)$ a surjective homomorphism of cdgas modelling the inclusion $\iota$. 
    Then, for all $k \geq 1$, we have
    \begin{equation*}
        \pi_k(\hAut_A(X),\id)\otimes \bbK \cong H_k(\Der(\Lambda V_X,\ker \mathsf{A}(\iota)),\delta).
    \end{equation*}
\end{thmx}

\vspace{10pt}
\noindent
\textbf{Outline of the article:}
In Section \ref{Section: Preliminaries}, we recall the essential facts and notational conventions of rational and real homotopy theory, and compute the algebraic models of the examples  we need later in the article. 
Section \ref{Section: App hAut via RHT} is devoted to proving Theorem \ref{Main Thm: Homotopy Automorphism via Derivations} and its generalisation, Theorem \ref{thm: hAut via derivations (general)}.
This result will be central for giving a convenient description of the classifying maps of the fibre bundles $D\mathcal{O}(k) \rightarrow \mathcal{DO}(k) \rightarrow S^2$ in terms of derivations on the minimal models of the fibres in Subsection \ref{Subsection: Top Properties Blow Up Families}, which will be used in Subsection \ref{subsection: resolving orbifolds} to prove Theorem \ref{Main thm: Orbifold Robustness}.
However, the article is written so that readers may assume Theorem \ref{Main Thm: Homotopy Automorphism via Derivations} as a black-box and immediately jump to the topological application in Section \ref{Section: RHT Orbifold Resolutions} and Section \ref{Section: G2 Applications} if they wish.
In Section \ref{Section: G2 Applications}, we prove Theorem \ref{Main Thm: G2 Manifold fibration detection} apply this theorem to the examples in \cite{Joyce1996CompactG2II}. 

\vspace{10pt}
\noindent
\textbf{Acknowledgments:} 
The author acknowledges support from the Australian Research Council Discovery Project DP220102163.
Furthermore, he would like to thank Diarmuid Crowley for continuing support and interest in this work.

\section{Preliminaries on Rational Homotopy Theory}\label{Section: Preliminaries}

%
%
%
%%%%%%%%%%%%%%%%%%%%%%%%%%%%%%%%%%%%%%%%%%%%%%%%%%%%%%%%%%%%%%%%%%%%%%%%%%%%
%             Subsection: Basics on rational homotopy theory                %
%%%%%%%%%%%%%%%%%%%%%%%%%%%%%%%%%%%%%%%%%%%%%%%%%%%%%%%%%%%%%%%%%%%%%%%%%%%%
%
%
%
We are going to recall some basic facts of rational homotopy theory that we need in this article.
A non-complete list of excellent sources consists of \cite{berglund2012RHT}, \cite{Felix2001RHT}, \cite{Bousfield1976PLdeRham}, and the original \cite{Sullivan1977Infinitesimal}.
Although we are going to apply the theory mostly to orbifolds and manifolds, we present it in full generality.
It will also be more convenient to work with simplicial sets instead of topological spaces.
This is unproblematic because the geometric realisation functor and the singular set functor form a Quillen equivalence $ |\placeholder| \colon \mathsf{Top} \rightleftarrows \mathsf{sSet} \colon S(\placeholder)$.
In particular, the two categories have isomorphic homotopy categories.
We first give a quick recollection of rational simplicial sets (or rational topological spaces) following \cite{Hilton1975Localisation} and then present the algebraic counterpart mostly following \cite{Bousfield1976PLdeRham}.

A connected Kan set is called \emph{nilpotent} if its fundamental group $\pi_1(X)$ is nilpotent and if the action of $\pi_1(X)$ on all higher homotopy groups is nilpotent.
A nilpotent Kan set is called \emph{rational} if, for all $n\geq 1$, the power maps $(\placeholder)^n \colon \pi_k(X) \rightarrow \pi_k(X)$ are isomorphisms for $k\geq 2$ and bijections for $k=1$.
A simplicial map $\ell_\Q \colon X \rightarrow X_\Q$ between Kan sets is called a \emph{rationalisation} if it satisfies the following universal property: For each rational Kan set and every simplicial map $g \colon X \rightarrow Y$ there exists a simplicial map $g_\Q$, unique up to homotopy, such that $g_\Q \circ \ell_\Q \simeq g$.
In particular, each simplicial map $f \colon X \rightarrow Y$ induces a unique (up to homotopy) simplicial map $f_\Q \colon X_\Q \rightarrow Y_\Q$ and the rationalisation of a Kan set is unique up to homotopy equivalence.
We call two Kan sets $X$ and $Y$ \emph{rational homotopy equivalent} if their rationalisations $X_\Q$ and $Y_\Q$ are (weakly) homotopy equivalent.
By Theorem 3B in Chapter II of \cite{Hilton1975Localisation}, the rationalisation induces an isomorphism $\pi_k(\ell_\Q) \colon \pi_k(X)\otimes\Q  \rightarrow \pi_k(X_\Q)$, where we use the Malcev completion for the fundamental group, which agrees with the usual tensor product with $\Q$ if $\pi_1(X)$ is abelian, see Chapter I of \cite{Hilton1975Localisation} for details.

On the algebraic side, recall that a \emph{commutative differential graded algebra} $(A,d)$ over a base field $\bbK$ of characteristic $0$ is a (commutative) group object in the category of chain complexes over $\bbK$.
It should be thought of as an abstraction of the de Rham complex of a smooth manifold.
A homomorphism of (commutative) differential graded algebras (\emph{dga-homomorphism} for short) is a chain map that respects the multiplicative structure and the unit.
Together, they form a category $\mathsf{CDGA}$.
We would like to emphasise that we do \emph{not} exclude the possibility that $1 = 0$; this happens if and only if $A = 0$ is the zero-algebra, which is the terminal object in the category $\mathrm{CDGA}$.

To each simplicial set $X$, we can assign the cdga $\Omega_{PL}^\bbK(X)$ of polynomial differential forms over $\bbK$, see \cite[Definition 2.1]{Bousfield1976PLdeRham} for details.
For example, for the (combinatorial) $n$-simplex, we have
\begin{equation*}
    \Omega_{PL}^\bbK(\Delta[n]) = \frac{\bbK[T_0,\dots,T_n] \otimes \Lambda[dT_0,\dots,dT_n]}{\langle T_0 + \dots + T_n = 1, dT_0 + \dots + dT_n =0 \rangle} 
\end{equation*}
with the differential defined as on the smooth differential forms.
Usually, we drop $\bbK$ and the subscript $PL$ from the notation if it does lead to confusion.
The algebra over polynomial differential forms gives rise to a contravariant functor that exhibit an adjunction
\begin{equation}\label{eq: spatial de Rham adjunction}
    \langle \placeholder \rangle \colon \mathsf{CDGA} \rightleftarrows \mathsf{sSet}^\op : \Omega(\placeholder),
\end{equation}
where the left adjoint functor $\langle \placeholder \rangle$ is referred to as the \emph{spatial realisation}. 
It is defined by $\langle B \rangle_n = \mathrm{Hom}_{dga}(B,\Omega(\Delta[n]))$.

The cdga $\Omega(X)$ satisfies a form of the de Rham theorem in the sense that $H(\Omega(X),d) \cong H_{\mathrm{simp}}(X;\bbK) \cong H_{sing}(|X|;\bbK)$, see \cite[Theorem 2.2]{Bousfield1976PLdeRham}.
To extract (topological) information about $X$ beyond its cohomology-ring, we need to replace the rather unwieldy algebra $\Omega(X)$ by more manageable cdgas.
To this end, recall that a \emph{quasi-isomorphism} between two cdgas $B_1$ and $B_2$ is a dga-homomorphism $\varphi \colon B_1 \rightarrow B_2$ that induces an isomorphism between their cohomology groups.
\begin{definition}
    An (algebraic) \emph{model} (or $\bbK$-\emph{model} if we wish to emphasise the underlying field) for a cdga $B$ is a pair $(C,m)$ consisting of another cdga $C$ together with a quasi-isomorphism $m \colon C \rightarrow B$.
    If $B_j$ is modelled by $(C_j,m_j)$ for $j=1,2$, we call $\psi$ a \emph{model} for the dga-homomorphism $\varphi\colon B_1 \rightarrow B_2$ if the following diagram commutes up to homotopy:
    \begin{equation*}
        \xymatrix@R-.4em{ C_1 \ar[rr]^\psi\ar[d]_{m_1} && C_2 \ar[d]^{m_2} \\
        B_1 \ar[rr]_\varphi && B_2.}
    \end{equation*}
    We call the models \emph{strict}, if the diagram commutes on the nose.
    A  model for $X$ is a model for $\Omega(X)$ and a model for a continuous map $f \colon X_1 \rightarrow X_2$ is a $\bbK$-model for $\Omega(f)$.
    If $\bbK$ is either $\Q$ or $\R$, we refer to them as rational or real models.
\end{definition}
We need the definition in this generality only in very few places. 
Mostly, the underlying cdga of our models will belong to the following subclass, which should be thought of as an algebraic analogue to the subclass of CW-complexes inside the category of topological spaces.
The following definition is borrowed from \cite{Felix2001RHT}.
\begin{definition}\label{def: Sullivan model}
    A cdga $(B,d)$ is called a \emph{Sullivan algebra} if there is a graded $\bbK$-vector space $V$ concentrated in non-negative degrees that has an ascending filtration $V(0) \subseteq V(1) \subseteq \dots$ such that the following three conditions are satisfied:
    \begin{align*}
        (B,d) &= (\Lambda V,d),  \qquad \bigcup_{p=0}^\infty V(p) = V, \qquad \text{and} \qquad dV(p) \subseteq \Lambda V(p-1), \ dV(0) = \{0\}.  
    \end{align*}
    A Sullivan algebra $(\Lambda V,d)$ is called \emph{minimal} if $d(V) \subseteq \Lambda^{\geq 2} V$, i,e., if for all $v \in V$ the element $d(v)$ is a linear combination of decomposable elements.

    A model $(C,m)$ for a cdga algebra $B$ is called a \emph{Sullivan model} or a \emph{minimal model} if the underlying cdga $C$ is a Sullivan algebra or a minimal algebra, respectively.
\end{definition}
We wish to emphasise, as in the case of CW-complexes, that the filtration is not part of the structure.
\begin{example}\label{emxpl: postnikov decomposition of Sullivan algebras}
    Let $(\Lambda V,d)$ be a minimal Sullivan algebra.
    Then the filtration $V(q) = V^{\leq q} := \bigoplus_{s\leq q} V^s$ is a filtration of $V$.
    It is easy to see that this filtration turns the differential graded subalgebra $(\Lambda V^{\leq p},d)$ of $(\Lambda V,d)$ into a (minimal) Sullivan algebra as well.
    In resemblance to the Postnikov decomposition of a nilpotent topological space, we denote the \emph{$p$-truncation} $(\Lambda V^{\leq p},d)$ by $\mathsf{P}_p\Lambda V$.
\end{example}
An important example of a Sullivan algebra that is not minimal is the `algebraic interval' $I = \Lambda[t,dt]$, which is generated by the graded vector space $V = V^0 \oplus V^1 = \R\cdot t \oplus \R\cdot dt$ and whose differential satisfies the tautological relation $d(t) = dt$.
It comes with two dga-homomorphism $\ev_0, \ev_1 \colon \Lambda[t,dt] \rightarrow \R$ that satisfy $\mathrm{ev}_j(t)=j$. 
Note, in particular, that there is an isomorphism $\Lambda[t,dt] \cong \Omega(\Delta[1])$ induced by $T_0 \mapsto t$ and $T_1 \mapsto 1-t$.
The algebraic interval allows us to define notion of a homotopy for dga-homomorphisms.
\begin{definition}\label{def: Algebra Homotopy}
    Two dga-homomorphisms $\varphi_0,\varphi_1 \colon B_1 \rightarrow B_2$ are \emph{homotopic} if there is a dga-homomorphism $H \colon B_1 \rightarrow \Lambda[t,dt] \otimes B_2$ such that $\mathrm{ev}_j \circ H = \varphi_j$.
\end{definition}
The importance of Sullivan algebras $\Lambda V$ is that being homotopic induces an equivalence relation on $\mathrm{Hom}_{dga}(\Lambda V,B)$, \cite[Proposition 12.7]{Felix2001RHT}.
Moreover, composition with a quasi-isomorphism $\varphi \colon B_1 \rightarrow B_2$ induces a bijection on the set of homotopy classes $[\Lambda V,B_1] \xrightarrow{\cong} [\Lambda V,B_2]$, see \cite[Proposition 12.9]{Felix2001RHT}.

By \cite[Theorem 14.12]{Felix2001RHT}, every cdga $B$ with $H^0(B) = \bbK$ has a minimal model and, for two minimal models $m_j \colon C_j \rightarrow B$, there is an isomorphism $\Phi \colon C_1 \rightarrow C_2$ such that $m_1$ and $m_2 \circ \Phi$ are homotopic.
For this reason, we will denote the minimal model of a simplicial set $X$ (or topological space) with $\mathsf{M}_X$.
Moreover, each dga-homomorphism $\varphi \colon B_1 \rightarrow B_2$ gives rise to a homomorphism $\mathsf{M}(\varphi) \colon \mathsf{M}_1 \rightarrow \mathsf{M}_2$ that is unique up to homotopy.

It was proved by Bousfield and Gugenheim that the category $\mathsf{CDGA}$ carries a closed model structure with quasi-isomorphisms as weak-equivalences and surjective dga-homomorphism as fibrations, see \cite[Theorem 4.3]{Bousfield1976PLdeRham}.
In this model structure, Sullivan algebras are cofibrant objects, which can be derived from \cite[4.5 Closure Properties]{Bousfield1976PLdeRham}, but see \cite[Theorem 8.11]{berglund2012RHT} for an explicit proof.

The adjunction (\ref{eq: spatial de Rham adjunction}) gives rise to a Quillen adjunction, but it fails to induce an equivalence between their homotopy categories (even if one only considers rational simplicial sets). %, because the unit $C \rightarrow \Omega \langle C \rangle$ fails to be a quasi-isomorphism in general.
However, if $\bbK = \Q$ and if $\mathrm{fin}_{\Q,\mathrm{nil}}\text{-}\mathrm{Ho}(\mathsf{sSet})$ denotes the full subcategory of the homotopy category of $\mathsf{sSet}$ whose objects are connected, nilpotent, Kan complexes of finite $\Q$-type (i,e., $H_i(X;\Q)$ is finite dimensional for all $i\geq0$) and if $\mathrm{fin}_\Q\text{-}\mathrm{Ho}(\mathsf{CDGA})$ denotes the full subcategory of $\mathrm{Ho}(\mathsf{CDGA})$ whose objects are connected cdgas (i,e., $H^0(C,d) = \Q$) of finite $\Q$-type, then Bousfield and Gugenheim proved that the spatial realisation and the minimal model induce an equivalence of categories
\begin{equation*}
    \langle \placeholder \rangle \colon \mathrm{fin}_\Q\text{-}\mathrm{Ho}(\mathsf{CDGA}) \overset{\cong}{\rightleftarrows} \mathrm{fin}_{\Q,\mathrm{nil}}\text{-}\mathrm{Ho}(\mathsf{sSet})^\op : \mathsf{M}, 
\end{equation*}
see \cite[Theorem 9.4]{Bousfield1976PLdeRham}.

Unpacking the notion of homotopy categories leads to the following explicit consequences.
\begin{prop}
    If $X_j$ with $j=0,1$ are nilpotent Kan sets (or CW-complexes) whose homology groups $H_\ast(X_j;\Q)$ are finite dimensional in every degree, then the following statements hold true:
    \begin{itemize}
        \item[(1)] $X_0$ and $X_1$ are rational homotopy equivalent if and only if their minimal models are isomorphic.
        \item[(2)] If, in addition, $X_0$ and $X_1$ are rational, then the set of homotopy classes agrees with the set of homotopy classes between their Sullivan models:
        \begin{equation*}
            [X,Y] \cong [\Lambda V_Y,\Lambda V_X]
        \end{equation*}
        \item[(3)] If $\mathsf{M}_X$ is the minimal model, then the map $\ell \colon X \mapsto \langle \mathsf{M}_X\rangle$, which should be thought of as the unit of the adjunction (\ref{eq: spatial de Rham adjunction}), is a rationalisation.
    \end{itemize}
\end{prop}
The minimal model of a nilpotent Kan set (over each field $\bbK$ of characteristic zero) is closely connected to its homotopy group.
\begin{prop}[Theorem 11.3 in \cite{Bousfield1976PLdeRham}]
    If $X$ is a nilpotent Kan complex of finite $\Q$-type and $\mathsf{M}_X = \Lambda V_X$ is minimal model, then there is a natural bijection $\pi_k(X) \otimes \bbK \cong \mathrm{Hom}(V^k_X,\bbK)$, which is a homomorphism whenever $\pi_k(X)$ is abelian.
\end{prop}
%%% In the proof of Theorem 11.3, they state that \pi_j(X) \otimes \Q is finite dimensional, so taking duals is unproblematic.

Since we will mostly deal with manifolds and orbifolds, it is more convenient to work out the examples over the real numbers using the de Rham complex of smooth differential forms.
The next result was stated in \cite[p.135ff]{Felix2001RHT} for smooth manifolds, but the proof given carries over to orbifolds.
%The proof strategy of the proof the next result given in \cite[p.135ff]{Felix2001RHT} generalises to orbifolds in the sense of \cite{Satake1956Orbifolds}.
%
\begin{prop}\label{prop: smooth de Rham theorem}
    Let $M$ be a smooth manifold or a smooth orbifold in the sense of \cite{Verona1988deRhamOrbitSpaces}, then $\Omega_{PL}^\R(M)$ and $\Omega_{dR}(M)$ are quasi-isomorphic.
\end{prop}
In particular, if the manifold $M$ is a closed (or more general if it is $\Q$-finite), then $\mathsf{M}_M^\R = \mathsf{M}^{\Q}_M \otimes \R$ and we can derive the real minimal model from the de Rham algebra.
We use this insight to close this section with a collection of examples we need in forthcoming sections.
\begin{example}\label{ex: Spheres and lens spaces}
    If $n$ is even, then the minimal model of the $n$-dimensional sphere $S^n$ is given by $\mathsf{M}_{S^n} = \Lambda[a_n,b_{2n-1}\,|\, db_{2n-1} = a_n^2]$.
    The model map $m \colon \mathsf{M}_{S^n} \rightarrow \Omega(S^n)$ sends $a_n$ to `the' volume form $\vol_{S^n}$ and $b_{2n-1}$ to zero.
    
    If $n$ is odd, then the minimal model of the $n$-dimensional sphere $S^n$, or more generally, of a lens space $S^n/\Z_k$, is given by $\mathsf{M}_{S^n/\Z_k} = \Lambda[a_n]$ with $\mathrm{deg}(a_n) = n$ and vanishing differential.
    The model map $m\colon \mathsf{M}_{S^n/\Z_k} \rightarrow \Omega(S^n/\Z_k)$ sends $a_n$ to `the' volume form $\mathrm{vol}_{S^n/\Z_k}$.
\end{example}
\begin{example}\label{exmpl: Rational model disc bundles}
    For all $k \in \Z$, the disc bundles of the complex line bundles $\mathcal{O}(k) \rightarrow \CP^1$ can be constructed from Hopf fibration via a Borel construction:
    \begin{equation*}
        D\mathcal{O}(k) = S^3 \times_{S^1,(\placeholder)^k} D^2 = (S^3 \times D^2)/\sim \qquad (p,\lambda) \sim (\mathrm{e}^{\iu\theta}p,\mathrm{e}^{k\iu\theta}\lambda). 
    \end{equation*}
    They are manifolds with boundary $S^3/\Z_k$ and they are all homotopy equivalent to $\CP^1$.
    Thus, their minimal models are given by $\mathsf{M}_{D\mathcal{O}(k)} = \Lambda[a_2,b_3 \,| \, db_3 = a_2^2]$.

    To model the boundary inclusion $\iota \colon \partial D\mathcal{O}(k) \hookrightarrow D\mathcal{O}(k)$, we choose the following dga-homomorphism $m \colon \mathsf{M}_{D\mathcal{O}(k)} \rightarrow \Omega_{dR}(D\mathcal{O}(k))$:
    The generator $a_2$ gets mapped to a Thom form $\tau$, which is a compactly supported $2$-form on $D\mathcal{O}(k) \setminus S^3/\Z_k$ whose integral over each fibre $D^2 \hookrightarrow D\mathcal{O}(k)$ is $1$.
    By the Thom isomorphism theorem, we have $\tau^2 = \mathrm{e}(\mathcal{O}(k)) \vol_{D\mathcal{O}(k)} = k \cdot \vol_{D\mathcal{O}(k)}$, where $\vol_{D\mathcal{O}(k)}$ is a four-form on $D\mathcal{O}(k)$ whose support does not intersect the boundary and that integrates to $1$ over $D\mathcal{O}(k)$.
    Let $\tilde{\omega} \in \Omega^3(D\mathcal{O}(k)$ be a primitive of $\vol_{D\mathcal{O}(k)}$ and $\eta \in \Omega^2(S^3/\Z_k)$ such that 
    \begin{equation*}
        d\eta = \tilde{\omega}|_{S^3/\Z_k} - \left(\int_{S^3/\Z_k} \tilde{\omega} \right) \cdot \vol_{S^3/\Z_k} = \tilde{\omega}|_{S^3/\Z_k} - \left(\int_{D^4/\Z_k} \vol_{D^4/\Z_k} \right) \cdot \vol_{S^3/\Z_k} = \tilde{\omega}|_{S^3/\Z_k} - \vol_{S^3/\Z_k} ,
    \end{equation*}
    with $\vol_{S^3/\Z_k}$ the generator from the previous example.
    Let, moreover, $\chi \colon D\mathcal{O}(k) \rightarrow [0,1]$ be a cut-off function that is identically $1$ near the boundary and whose support is disjoint from the one of $\vol_{D\mathcal{O}(k)}$.
    We now define $m(b_3) = k\omega := k(\tilde{\omega} - d(\chi \eta))$, which is still a primitive of $\tau^2$, but now restricts to the differential form $k\cdot \vol_{S^3/\Z_3}$ from the previous example.

    With this model map at hand, it is clear now that the dga-homomorphism $\Omega(\iota)$ is modelled by the maps 
    \begin{equation*}
        \mathsf{M}_{D\mathcal{O}(k)} = \Lambda[a_2,b_3] \rightarrow \mathsf{M}_{\partial D\mathcal{O}(k)} = \Lambda[\beta_3], \qquad \qquad  b_3 \mapsto k\cdot \beta_3. 
    \end{equation*}
\end{example}
\begin{rem}
    Since the composition $S^3 \hookrightarrow \mathcal{O}(-1) \rightarrow \CP^1$ is the Hopf fibration, the previous example also shows that the Hopf fibration is modelled by the dga-homomorphism $\mathsf{M}_{S^3}=\Lambda[a_2,b_3] \rightarrow \Lambda[\beta_3] = \mathsf{M}_{S^2}$ that is given by $b_3 \mapsto \beta_3$.
\end{rem}
\begin{example}\label{ex: projective spaces}
    The minimal model of a complex projective space is given by the cdga $\mathsf{M}_{\CP^n} = \Lambda[ a_2,b_{2n+1} \, | \, db_{2n+1} = a_{2}^{n+1}]$.
    The model map $m \colon \mathsf{M}_{\CP^n} \rightarrow \Omega(\CP^n)$ sends $a_2$ to $\omega_{FS}/\sqrt{2}$, the K\"ahler form of the Fubini-Study metric and $b_{2n+1}$ to zero.
\end{example}
The next example discussion rational models of (homotopy) push outs of simplicial sets (or CW complexes).
Details can be found in \cite[Section 13.1]{Felix2001RHT}.
\begin{example}\label{exmpl: pushout model}
   Let $B_1 \xrightarrow{\varphi_1} B_0 \xleftarrow{\varphi_2} B_2$ be dga-homomorphisms between two cdgas.
   Its \emph{fibre product} $B_1 \oplus_{B_0} B_2$ is the cdga
   \begin{equation*}
       B_1 \oplus_{B_0} B_2 = \bigl\{ (x,y) \, : \, \varphi_1(x) = \varphi_2(y) \bigr\} \subseteq B_1 \oplus B_2,
   \end{equation*}
   where the differential $d$ acts component-wise.

   Let $X_0$, $X_1$, $X_2$ be simplicial sets and let $\iota_j \colon X_0 \rightarrow X_j$ be simplicial maps with one map being a cofibration (i,e., injective). 
   If one of the models $\mathsf{A}(\iota_j) \colon \mathsf{A}(X_j) \rightarrow \mathsf{A}(X_0)$ is surjective, then the fibre product serves $\mathsf{A}(X_1) \oplus_{\mathsf{A}(X_0)} \mathsf{A}(X_2)$ is a rational model for the (homotopy) push out $Y$ of $X_1 \xleftarrow{\iota_1} X_0 \xrightarrow{\iota_2} X_2$. 
\end{example}

Besides allowing for a calculation of real homotopy groups out of the de Rham complex, we can use the minimal model of a compact manifold $M$ with boundary $\partial M$ to get information about the rational homotopy groups of its homotopy automorphisms (relative to the boundary) using Theorem \ref{Main Thm: Homotopy Automorphism via Derivations}.
As a demonstration, we present these calculations for the manifolds $D\mathcal{O}(k)$ with boundary $S^3/\Z_k$ and our results to the real homotopy groups of the topological monoid $\hAut(D\mathcal{O}(k))$ of `absolute' homotopy equivalences.

\begin{example}\label{exmpl: hAut of disc bundles}
    We have seen in Example \ref{exmpl: Rational model disc bundles} that the  minimal models of $D\mathcal{O}(k)$ and its boundary $S^3/\Z_k$ are given by $\mathsf{M}_{D\mathcal{O}(k)} = \Lambda[a_2,b_3 \, |\, db_3 = a_2^2]$ and $\mathsf{M}_{S^3/\Z_k} = \Lambda[\beta_3]$. 
    Moreover, it was shown that the real model for the boundary inclusion is modelled by the dga-homomorphism $\mathsf{M}(\iota) \colon \mathsf{M}_{D\mathcal{O}(k)} \rightarrow \mathsf{M}_{S^3/\Z_k}$ that send $a_2$ to $0$ and $b_3$ to $\beta_3$.
    In particular, $\ker \mathsf{M}(\iota) = \bigoplus_{n\neq 0,3} \mathsf{M}^n_{D\mathcal{O}(k)}$.

    Since (graded) derivations from a free algebra are completely determined by the image of its generating vector space, we observe that
    \begin{equation*}
        \Der_n(\mathsf{M}_{D\mathcal{O}(k)}, \mathsf{M}_{D\mathcal{O}(k)}) = \begin{cases}
            \R, & \text{if } n\leq 3, \\
            0, & \text{if }n \geq 4,
        \end{cases} \quad \text{while} \quad 
        \Der_n(\mathsf{M}_{D\mathcal{O}(k)}, \ker \mathsf{M}(\iota)) =  \begin{cases}
            \R, & \text{if } n=1, \\
            0, & \text{if }n \geq 2. 
        \end{cases}
    \end{equation*}
    with $\Der(\mathsf{M}_{D\mathcal{O}(k)}, \mathsf{M}_{D\mathcal{O}(k)})$ generated by the linear maps $a_2 \otimes b_3^\vee$, $1 \otimes a_2^\vee$, and $1\otimes b_3^\vee$ of degree $1$, $2$, and $3$.
    Observe that the latter two linear maps take non-zero values in $\mathsf{M}_{D\mathcal{O}(k)}^0 = \R\cdot 1$, which does not lie in the differential ideal $\ker \mathsf{M}(\iota)$.

    A straight-forward computation shows that
    \begin{equation*}
        \delta_1(a_2\otimes b_3^\vee) = 0, \quad \delta_2(1\otimes a_2^\vee) = -2a_2\otimes b_3^\vee, \quad \text{and} \quad \delta_3(1\otimes b_3^\vee) = 0, 
    \end{equation*}
    and Theorem \ref{Main Thm: Homotopy Automorphism via Derivations} now implies that 
    \begin{equation*}
        \pi_n(\hAut(D\mathcal{O}(k)),\id)\otimes \R \cong H_n(\Der(\mathsf{M}_{D\mathcal{O}(k)},\mathsf{M}_{D\mathcal{O}(k)}),\delta) = \begin{cases}
            \R, & \text{if } n=3, \\
            0, & \text{if } n\neq 3, 
        \end{cases} 
    \end{equation*}
    while 
    \begin{equation*}
        \pi_n(\hAut_\partial(D\mathcal{O}(k)),\id)\otimes \R \cong H_n(\Der(\mathsf{M}_{D\mathcal{O}(k)},\ker \mathsf{M}(\iota)),\delta) = \begin{cases}
            \R, & \text{if } n=1, \\
            0, & \text{if } n\neq 1. 
        \end{cases}    
    \end{equation*}
\end{example}
\section{Algebraic models for homotopy automorphisms}\label{Section: App hAut via RHT}

In this section, we are going to study mapping spaces using rational homotopy theory. 
The main goal of this section is to give a detailed proof of Theorem \ref{Main Thm: Homotopy Automorphism via Derivations} and its generalisation, Theorem \ref{thm: hAut via derivations (general)} below, which state that the rational homotopy groups of a (relative) mapping space $\mathrm{C}_f(X,Y)$ can be computed in terms of derivations between the rational models of $X$ and $Y$.

The presentation closely follows the philosophy of \cite{Bousfield1976PLdeRham} by mostly relying on homotopical algebra.
In fact, we are going to generalise Section 5.8 in loc. cit. from augmentations to general surjective dga-homomorphisms.
We first explore the relation between the topological and simplicial mapping spaces to their algebraic counterparts and, in a second step, relate the homotopy groups of the algebraic mapping spaces to appropriate homology groups of derivations.

\subsubsection*{Topological and Simplicial Mapping Spaces}

Although the category of simplicial sets and topological spaces are Quillen equivalent (via the geometric realisation functor $|\placeholder|$ and the singular set functor $S(\placeholder)$ ), it is worthwhile to recall their precise relation.

For two (compactly generated Hausdorff) spaces $X$ and $Y$, let $\mathrm{C}(X,Y)$ be the space of all continuous maps equipped with the kellification of the compact-open topology; that is a subset of $\mathrm{C}(X,Y)$ is closed if and only if $A \cap C$ is closed with respect to the compact open topology for every compact subspace $C \subseteq \mathrm{C}(X,Y)$.
In this way, mapping spaces exhibit an exponential law $\mathrm{C}(X \times Y,Z) \cong \mathrm{C}(X,\mathrm{C}(Y,Z))$.
If $\iota \colon A \hookrightarrow X$ is a cofibration, then $|_A = (\placeholder)\circ \iota_A \colon \mathrm{C}(X,Y) \rightarrow \mathrm{C}(A,Y)$ is a Serre fibration and the fibre of $f\colon A \rightarrow Y$ is denoted by $\mathrm{C}_f(X,Y)$ or $\mathrm{C}_{A,f}(X,Y)$ if we wish to emphasise the subspace as well. 
If $X=Y$ and $f=\id$, we may simply write $\mathrm{C}_A(X)$.

For two simplicial sets $X$ and $Y$, the simplicial analogue of a mapping space is the simplicial set $\mathrm{map}(X,Y)$ whose set of $n$-simplicies is given by the set
\begin{equation*}
    \mathrm{map}_n(X,Y) = \mathrm{Hom}_{\mathsf{sSet}}(\Delta[n] \times X,Y)
\end{equation*}
and the structure maps are precomposition with $\phi \times \id_X$ where $\phi \colon \Delta[m] \rightarrow \Delta[n]$ is a morphism in the simplex category.

It is known that $\mathrm{map}(X,Y)$ is a Kan set if the target $Y$ is a Kan set, see \cite[Prop 1.17]{curtis1971simplicial}, and that $\mathrm{map}(X,Y)$ is the internal right-adjoint functor to the product functor \cite[Proposition I.5.1]{goerss2009simplicial}, which means that, for all simplicial set $Z$, there exists a natural bijection, the `simplicial exponential law',
\begin{equation*}
    \mathrm{Hom}_{\mathsf{sSet}}(Z \times X,Y) \cong \mathrm{Hom}_{\mathsf{sSet}}\left(Z,\mathrm{map}(X,Y)\right).
\end{equation*}
If $X,Y$, and $Z$ are simplicial sets, then the composition of simplicial maps extends to a simplicial map
\begin{equation*}
    \mathrm{map}(X,Y) \times \mathrm{map}(Y,Z) \rightarrow \mathrm{map}(X,Z).
\end{equation*}
Explicitly, the composition $f \circ g$ of $g \in \mathrm{map}_n(X,Y) = \mathrm{Hom}_{\mathsf{sSet}}(\Delta[n] \times X,Y)$ with $f \in \mathrm{map}_n(Y,Z) = \mathrm{Hom}_{\mathsf{sSet}}(\Delta[n] \times Y,Z)$ is given by
\begin{equation*}
    \xymatrix{ \Delta[n] \times X \ar[rr]^-{(\id_{\Delta[n]},g)} && \Delta[n] \times Y \ar[rr]^-f && Z.}
\end{equation*}
For every topological space $X$, the counit $\mathrm{ev} \colon |S(X)| \rightarrow X$ is a weak homotopy equivalence, where $S(X)$ denotes the singular set of $X$, see \cite[p.207]{curtis1971simplicial}.
Furthermore, observe that the graph construction yields a simplicial map
\begin{equation*}
    S(\mathrm{C}(X,Y)) \rightarrow \mathrm{map}(S(X),S(Y)) 
\end{equation*}
that sends an element $f \colon \Delta^n \rightarrow \mathrm{C}(X,Y)$ of $S_n(\mathrm{C}(X,Y))$ with adjoint map $\mathrm{Ad}(f) \colon X \times \Delta^n \rightarrow X$ to the simplicial map $S(X) \times \Delta[n] \rightarrow S(X)$ that is given by
\begin{equation*}
     S_m(X) \times \Delta[n]_m \ni (\sigma,\varphi) \mapsto \bigl( \Delta^m \xrightarrow{ (\sigma, \varphi) } X \times \Delta^n \xrightarrow{\mathrm{Ad}(f)} X \bigr) \in S_m(X).
\end{equation*}
Since the two simplicial sets $S\mathrm{C}(X,Y)$ and $\mathrm{map}(S(X),S(Y))$ are Kan, it is easy to see using the combinatorial description of their homotopy groups that this map is a weak homotopy equivalence and induces the canonical identification\footnote{Of course, this can also be deduced from the exponential law.}
\begin{equation*}
    \pi_n(\mathrm{C}(X,Y),f) \cong \pi_0\bigl(  \mathrm{C}_{\{\ast\}\times X,f}(S^n \times X, X)\bigr). %= [ X \times S^n, X \times \ast; X]_f 
\end{equation*}
We thus have a zig-zag of weak homotopy equivalences 
\begin{equation}\label{eq: zig-zag of topolgical mapping spaces}
    |\mathrm{map}(S(X),S(Y))| \xleftarrow{\simeq} |S(\mathrm{C}(X,Y))| \xrightarrow{\simeq} \mathrm{C}(X,Y), 
\end{equation}
and so we are allowed to work with $S(\mathrm{C}(X,Y))$, if we desire to apply simplicial methods.
The geometric realisation turns a Kan-fibration into a Serre fibration, see \cite[Theorem 10.10]{goerss2009simplicial}, and the singular set functor turns a Serre fibration into a Kan fibration.
Since this zig-zag (\ref{eq: zig-zag of topolgical mapping spaces}) is natural in both components, we also get a zig-zag for the relative mapping spaces:
\begin{equation}\label{eq: zig-zag of relative topolgical mapping spaces}
    |\mathrm{map}_{Sf}(S(X),S(Y))| \xleftarrow{\simeq} |S(\mathrm{C}_f(X,Y))| \xrightarrow{\simeq} \mathrm{C}_f(X,Y). 
\end{equation}
\subsubsection*{Algebraic Mapping Spaces}
Following Bousfield and Gugenheim \cite{Bousfield1976PLdeRham}, we construct the counterpart of mapping spaces in the category of commutative differential graded algebras $\mathsf{CDGA}$ over a field $\bbK$ of characteristic zero.
\begin{definition}
    For two cdgas $B_0$ and $B_1$, the simplicial set $\mathrm{map}(B_1,B_2)$ is defined as follows:
    The set of all $n$-simplices is
    \begin{equation*}
        \mathrm{map}_n(B_1,B_2) := \mathrm{Hom}_{\mathsf{sSet}}\bigl(B_1,\Omega(\Delta[n])\otimes B_2\bigr)
    \end{equation*}
    and face and degeneracy maps are given by composition with $\Omega(\phi) \otimes \id_{B_2}$ where $\phi \colon \Delta[m] \rightarrow \Delta[n]$ is a morphism in the simplex category.
\end{definition}
Note that this is a generalisation of the spatial realisation functor: $\langle B_1 \rangle = \mathrm{map}(B_1,\bbK)$.

We remind the reader that a cofibration in $\mathsf{CDGA}$ is a dga-homomorphism $C_1 \rightarrow C_2$ that has the left lifting property with respect to all surjective quasi-isomorphism, i,e., for each commutative outer square exists a dashed filler:
\begin{equation*}
    \xymatrix@R-1em{ C_1 \ar[rr] \ar@{^{(}->}[d]_\iota && B_1 \ar@{->>}[d]^\simeq \\
    C_2 \ar[rr] \ar@{-->}[rru] && B_2. }
\end{equation*}
If $C_1 = \bbK\cdot 1$ and $\iota$ is the unit map, then $C_2$ is called cofibrant, and Sullivan algebras are examples of such.

We will derive many of our results from the following `fundamental result'.
\begin{theorem}\cite[Theorem 5.3]{Bousfield1976PLdeRham}\label{thm: fundamental theorem}
    Let $\iota \colon C_1 \rightarrow C_2$ be a (dga-)cofibration and let $p \colon B \rightarrow T$ be a surjective dga-homomorphism.
    Then the natural simplicial map
    \begin{equation*}
        \iota \ast p = \bigl( (\placeholder)\circ\iota, p \circ (\placeholder) \bigr) \colon \mathrm{map}(C_2,B) \rightarrow \mathrm{map}(C_1,B) \times_{\mathrm{map}(C_1,T)} \mathrm{map}(C_2,T)
    \end{equation*}
    is a Kan fibration.
    If, in addition, $\iota$ or $p$ is a quasi-isomorphism, then $\iota \ast p$ is a weak-equivalence.
\end{theorem}
We will mostly use the following special case where $C_1 = \bbK\cdot 1$.
\begin{cor}\label{cor: weak fundamental theorem}
    If $C$ is a cofibrant cdga, e.g. a Sullivan algebra, and $p \colon B \rightarrow T$ a surjective dga-homomorphism, then the map $p \circ (\placeholder) \colon \mathrm{map}(C,B)\rightarrow \mathrm{map}(C,T)$ is a Kan fibration. 
    If $p$ is additional a quasi-isomorphism, then $p\circ(\placeholder)$ is a weak equivalence.
\end{cor}
\begin{cor}
    Let $B$ and $C$ be two cdgas with $C$ cofibrant.
    Then $\mathrm{map}(C,B)$ is a Kan set.
    In particular, if $X$ a simplicial set, then the simplicial sets $\mathrm{map}(X,\langle C \rangle)$ and $\mathrm{map}(C,\Omega(X))$ are Kan sets.
\end{cor}
\begin{proof}
    By definition, $\iota \colon \bbK \hookrightarrow C$ is a cofibration, so by Theorem \ref{thm: fundamental theorem} the restriction $\mathrm{map}(C,\Omega(X)) \rightarrow \mathrm{map}(\bbK,B) = \{1_B\}$ is Kan fibration; hence $\mathrm{map}(C,B)$ is a Kan set. %(because dga-homomorphisms are assumed to be unital)  
    It follows, in particular, that $\langle C \rangle = \mathrm{map}(C,\bbK)$ is Kan, so the simplicial mapping space $\mathrm{map}(X,\langle C \rangle)$ is Kan, too.
\end{proof}

The next result should be thought of as a space-level version of Whitehead's theorem in the algebraic set-up.
\begin{prop}\cite[Proposition 5.7]{Bousfield1976PLdeRham}\label{prop: space Whitehead cdga}
    Let $C$ be a cofibrant cdga and let $f \colon B_1 \rightarrow B_2$ be a (not necessarily surjective) quasi-isomorphism.
    Then post composition with $f$ induces a weak equivalence $f \circ (\placeholder) \colon \mathrm{map}(C,B_1) \rightarrow \mathrm{map}(C,B_2)$.
\end{prop}

We now turn our attention to relative mapping sets that should be thought of as the algebraic counterparts to $\mathrm{C}_f(X,Y)$.
\begin{definition}
    Let $B_2 \twoheadrightarrow T$ be a surjective dga-homomorphism and let $\varphi \colon B_1 \rightarrow T$ be a fixed dga-homomorphism.
    Define the $\varphi$-relative mapping set $\mathrm{map}^{T,\varphi}(B_1,B_2)$ as the following pullback
    \begin{equation*}
        \xymatrix{  \mathrm{map}^{T,\varphi}(B_1,B_2) \ar[rr] \ar[d] && \mathrm{map}(B_1,B_2) \ar[d]^{p \circ (\placeholder)}  \\
        \ast \ar[rr]_\varphi && \mathrm{map}(B_1,T). }
    \end{equation*}
\end{definition}
If it does not lead to confusion, we will either drop $T$ or $\varphi$ from the decoration.
Note that, as $p \circ (\placeholder)$ is a Kan fibration, the weak homotopy type of $\mathrm{map}^{T,\varphi}(B_1,B_2)$ only depends on the homotopy class of $\varphi \colon B_1 \rightarrow T$.

The next results essentially says that we can replace cdgas with their rational models without changing the weak homotopy type of their relative mapping sets.
\begin{lemma}\label{lem: quasi-iso invariant fibres}
    Let $H \colon B_1 \rightarrow \Lambda[t,dt] \otimes T_2$ be a homotopy that makes the following diagram, in which the horizontal arrows are quasi-isomorphism,
    \begin{equation*}
        \xymatrix{ B_1 \ar[rr]^{F}_\simeq \ar@{->>}[d]_{p_1} && B_2 \ar@{->>}[d]^{p_2}\\ 
        T_1 \ar[rr]_f^\simeq && T_2 }
    \end{equation*}
    commutative up to homotopy. 
    Let $C$ be a cofibrant cdga, $\Phi \colon C \rightarrow B_1$ a dga-homomorphism, and set $\varphi = p_1 \circ \Phi$.
    Then the homotopy $H$ induces a weak equivalence $\mathrm{map}^{T_1,\varphi}(C,B_1) \rightarrow \mathrm{map}^{T_2,p_2 F\Phi}(C,B_2)$.
\end{lemma}
\begin{proof}
    %Since the proof is a formal consequence of Theorem \ref{thm: fundamental theorem} and standard arguments in homotopy theory we will only sketch its.

    By definition of the algebraic mapping space, the homotopy $H$ defines $1$-simplex inside $\mathrm{map}(B_1,T_2)$, which can be represented by a simplicial map $H \colon \Delta[1] \rightarrow \mathrm{map}(B_1,T_2)$.
    Composition with $H$ yields a homotopy $\mathrm{H} \colon \Delta[1] \times \mathrm{map}(C,B_1) \rightarrow \mathrm{map}(C,T_2)$, which makes the following diagram of Kan fibrations commutative up to homotopy
    \begin{equation*}
        \xymatrix@R-.4em{ \mathrm{map}(C,B_1) \ar[rr]^{\Phi \circ (\placeholder)}\ar[d]_{p_1\circ (\placeholder)} && \mathrm{map}(C,B_2) \ar[d]^{p_2 \circ (\placeholder)}\\
        \mathrm{map}(C,T_1) \ar[rr]_{f\circ (\placeholder) } && \mathrm{map}(C,T_2). }
    \end{equation*}
    Since right the vertical is a Kan-fibration, we can lift the homotopy $\mathrm{H} \colon \Delta[1] \times \mathrm{map}(C,B_1) \rightarrow \mathrm{map}(C,T_2)$  %%% the homotopy is given by composition.
    to a homotopy $\mathcal{H} \colon  \Delta[1] \times \mathrm{map}(C,B_1) \rightarrow \mathrm{map}(C,T_2)$ extending $\Phi$ on $\{1\} \times \mathrm{map}(C,B_1)$).
    If $G$ denotes the restriction of $\mathcal{H}$ to $\{0\} \times \mathrm{map}(C,B_1)$, then $G$ is a map of Kan fibrations (over $f$), so it induces a map of fibres $G \colon \mathrm{map}^{T_1,\varphi}(C,B_1) \rightarrow \mathrm{map}^{T_2,p_2G\Phi}(C,T_2)$.

    Since $H$ is a path inside the Kan set $\mathrm{map}(C,T_2)$ between $p_2G\Phi=f\varphi$ and $p_2F\Phi$, it gives rise to a weak equivalence $\hat{H} \colon \mathrm{map}^{T_2,p_2G\Phi}(C,B_2) \rightarrow \mathrm{map}^{T_2,pF\Phi}(C,B_2)$.
    
    Since $F$ and $f$ are quasi-isomorphism, $F\circ (\placeholder)$ and $f\circ (\placeholder)$ are weak equivalences by Proposition \ref{prop: space Whitehead cdga}. 
    As $G$ is homotopic to $F\circ (\placeholder)$, it must be a weak equivalence.
    By the Five lemma, its restriction of $G \circ (\placeholder)$ to the fibre $\mathrm{map}^{T_1,\varphi}(C,B_1)$ is a weak equivalence, too.
    Thus, the composition of $\hat{H} \circ G \circ (\placeholder)\colon \mathrm{map}^{T_1,\varphi}(C,B_1) \rightarrow \mathrm{map}^{T_2,pF\Phi}(C,B_2)$ is a weak equivalence we are looking for.
\end{proof}
The natural bijection $\mathrm{Hom}_{dga}(B,\Omega(X)) \cong \mathrm{Hom}_{\mathsf{sSet}}(X,\langle B\rangle)$ of the 
 adjunction (\ref{eq: spatial de Rham adjunction}) can be used to define a map of simplicial sets
\begin{equation*}
    \mu \colon \mathrm{map}(B,\Omega(X)) \rightarrow \mathrm{map}(X, \langle B\rangle).
\end{equation*}
On the level of $n$-simplices, the natural map given by (external) multiplication
    \begin{equation*}
        \mu_n \colon \Omega(X) \otimes \Omega(\Delta[n]) \rightarrow \Omega( X \times \Delta[n] ), \qquad (x,y) \mapsto \Omega(\mathrm{pr}_1)(x) \cdot \Omega(\mathrm{pr}_2)(y),
    \end{equation*}
    which is, in fact, a quasi-isomorphism of cdgas by the Künneth formula.
    Together with the following adjunctions, the external product yield a sequence of (natural) maps
    \begin{align*}
        \mathrm{map}_n\bigl(B,\Omega(X)\bigr) ={}& \mathrm{Hom}_{dga}\bigl(B,\Omega(X) \otimes \Omega(\Delta[n]) \bigr) \xrightarrow{\mu_n \circ (\placeholder)} \mathrm{Hom}_{dga}\bigl( B, \Omega(X \times \Delta[n]) \bigr) \\
        \cong{}& \mathrm{Hom}_\mathsf{sSet}(X \times \Delta[n], \langle B \rangle)\\
        ={}& \mathrm{map}_n(X,\langle B \rangle),
    \end{align*}
    and naturality of the adjunction implies that these maps form a simplicial map $\mu$.
\begin{prop}
    If $C$ is a cofibrant cdga, then the natural map $\mu \colon \mathrm{map}(C,\Omega(X)) \rightarrow \mathrm{map}(X,\langle C \rangle)$ is a weak homotopy equivalence.
\end{prop}
\begin{proof}
    Using the combinatorial description of homotopy groups for Kan sets, we see that $\pi_n(\mu)$ decomposes as follows:
    \begin{align*}
        \pi_n( \mathrm{map}(C,\Omega(X)),\varphi) = \pi_0\bigl( \mathrm{map}^{\Omega(X),\varphi}( C, \Omega(S^n) \otimes \Omega(X) ) \bigr) &\rightarrow \pi_0\bigl(\mathrm{map}^{\Omega(X),\varphi}(C,\Omega(S^n \times X)\bigr)\\
        &\rightarrow \pi_0( \mathrm{map}_{X,\varphi}\bigl( S^n \times X, \langle C\rangle ) \bigr) \\
        &= \pi_n( \mathrm{map}\bigl( X,\langle C \rangle ),\varphi \bigr).
    \end{align*}
    Since the (external) tensor product $\Omega(S^n)\otimes \Omega(X) \rightarrow \Omega(S^n \times X)$ is a quasi-isomorphism, the first map in the composition is an isomorphism by Theorem \ref{thm: fundamental theorem}.
    
    While the two simplicial sets $\mathrm{map}^{\Omega(X),\varphi}\bigl(C,\Omega(S^n \times X)\bigr)$ and $\mathrm{map}_{X,\varphi}\bigl( S^n \times X, \langle C\rangle\bigr)$ have the same $0$-simplices, the latter has more $1$-simplices.
    Hence, we need to show that the second map is injective, which we accomplish by 
    %In other words that if there is a `simplicial homotopy' $H \in \mathrm{Hom}\bigl(C,\Omega(\Delta[1] \times S^n \times X)\bigr)$ that relates $\varphi$ and $\psi$, then there is also an `algebraic' homotopy $H' \in \mathrm{Hom}_{\mathsf{sSet}}( C, \Omega(\Delta[1]) \otimes \Omega(S^n \times X) )$ relating them.
    generalising the proof of Lemma 8.43 in \cite{berglund2012RHT} to the relative set-up.
    Let $H \colon \Delta[1] \times S^n \times X \rightarrow \langle C \rangle$ be a homotopy between $\psi_0$ and $\psi_1$ that restricts to $\varphi \circ \mathrm{pr}_X$ on $\Delta[1] \times \ast \times X$.
    Under the adjunction of $\langle \placeholder \rangle$ and $\Omega(\placeholder)$, the homotopy $H$ corresponds to a dga-homomorphism $\mathrm{Ad}(H) \colon C \rightarrow \Omega(\Delta[1] \times S^n \times X)$ whose composition with the face maps $\partial^0$ and $\partial^1$ yields $\psi_0$ and $\psi_1$, respectively.

    Since face maps $\partial^j \colon \Omega(\Delta[1] \times S^n \times X) \rightarrow \Omega(S^n \times X)$ and the degeneracy map $\sigma_0 \colon \Omega(S^n \times X) \rightarrow \Omega(\Delta[1] \times S^n \times X)$ are quasi-isomorphisms, we obtain a zig-zag of weak equivalences
    \begin{equation*}
        \xymatrix@C-0.7em{ \mathrm{map}^{\Omega(X),\varphi}(C,S^n \times X) \ar@<-0.3pc>[r]_-{\sigma_0} & \mathrm{map}^{\Omega(\Delta[1] \times X),\varphi \circ \mathrm{pr}_X}\bigl(C,\Omega(\Delta[1] \times S^n \times X)\bigr) \ar@<0.3pc>[r]^-{\partial^1} \ar@<-0.3pc>[l]_-{\partial^0} & \mathrm{map}^{\Omega(X),\varphi}(C,S^n \times X) \ar@<0.3pc>[l]^-{\sigma_0}. }
    \end{equation*}
    Since $\partial^j\sigma_0 = \id$ and $\partial^j$ are weak equivalences, we deduce that $\pi_0(\sigma_0) $ is the inverse of $ \pi_0(\partial^j)$.
    Since $\mathrm{Ad}(H)$ is a $0$-simplex in $\mathrm{map}^{\Omega(\Delta[1] \times X),\varphi \circ \mathrm{pr}_X}\bigl(C,\Omega(\Delta[1] \times S^n \times X)\bigr)$, we deduce that $[\mathrm{Ad}(H)] = \pi_0(\sigma_0)([\psi_0])$ and hence that
    \begin{align*}
        [\psi_0] &= \pi_0(\partial^0)([\mathrm{Ad}(H)]) = \pi_0(\partial^0)\bigl( \pi_0(\sigma_0)([\psi_0]) \bigr) = \pi_0(\partial^1)\bigl( \pi_0(\sigma_0)([\psi_0]) \bigr) \\
        &= \pi_0(\partial^1)([\mathrm{Ad}(H)]) = [\psi_1] \in \pi_0(\mathrm{map}^{\Omega(X),\varphi}(C,S^n \times X)),
    \end{align*}
    so injectivity is proved.
\end{proof}
Because $\mu$ is natural, we immediately conclude the following consequnce.
\begin{cor}
    If $A \hookrightarrow X$ is an inclusion of simplicial sets and $\varphi \colon C \rightarrow \Omega(A)$ a fixed dga-homomorphism with $C$ a cofibrant cdga, then the natural map $\mu$ induces a weak equivalence
    \begin{equation*}
        \mathrm{map}^{\Omega(A),\varphi}(C,\Omega(X)) \rightarrow \mathrm{map}_{A,\varphi}(X,\langle C\rangle).
    \end{equation*}
\end{cor}
Now, let $K$ be a connected, nilpotent Kan set of finite $\Q$-type and $\Lambda V_K \rightarrow \Omega(K)$ a Sullivan model for $\Omega(K)$. 
Let further $\mathsf{A}(\iota) \colon \mathsf{A}(X) \twoheadrightarrow \mathsf{A}(A)$ a surjective model for the inclusion $\iota \colon A \hookrightarrow X$.
Let $\eta \colon K \rightarrow \langle\Lambda V_K\rangle$ be a lift of the unit map $K \rightarrow \langle \Omega(K) \rangle$, which is unique up to homotopy and induces an isomorphism on rational homotopy groups, see \cite[Theorem 11.2]{Bousfield1976PLdeRham}.
For $f \colon A \rightarrow K$, denote the composition $\eta \circ f$ with $\varphi$, which can be also interpreted as a dga-homomorphism $\varphi \colon \Lambda V_K \rightarrow \Omega(A)$.
Let further $\psi \colon \Lambda V_K \rightarrow \mathsf{A}(A)$ be a lift of $\varphi$ under the model map $\mathsf{A}(A) \rightarrow \Omega(A)$, which exists and is unique up to homotopy by Proposition \ref{prop: space Whitehead cdga}.
Then, by Lemma \ref{lem: quasi-iso invariant fibres} there is a zig-zag of weak homotopy equivalences
\begin{align}\label{eq: zig-zag of mapping spaces}
    \mathrm{map}_f(X,K) \xrightarrow{\eta \circ  (\placeholder)} \mathrm{map}_\varphi(X,\langle\Lambda V_K\rangle) \xleftarrow{\mu} \mathrm{map}^\varphi(\Lambda V_K,\Omega(X)) \leftarrow \mathrm{map}^\psi(\Lambda V_K,\mathsf{A}(X)).
\end{align}

Note furthermore that there are cdga-homomorphisms $(H(S^n),0) \rightarrow \Omega(S^n)$ that induces the identity on cohomology, and all of them are homotopic to each other.
Together with the combinatorial homotopy group description of Kan sets, we can prove the next result.
%\textcolor{red}{Something with the addition is still fishy here.}
\begin{cor}\label{cor: Corollary of relative adjunction}
    Let $f$, $\varphi$, and $\psi$ be as above. 
    Let $F \colon X \rightarrow K$ an extension of $f$ and $\Psi \colon \Lambda V_K \rightarrow \mathsf{A}(X)$ be a model for $\eta \circ F$.
    Then the zig-zag (\ref{eq: zig-zag of mapping spaces}) together with each quasi-isomorphism $(H(S^n),0) \rightarrow \Omega(S^n)$ induces a homomorphism
    \begin{equation*}
        \pi_n( \mathrm{map}_f(X,K),F) \rightarrow \pi_n( \mathrm{map}^\psi(\Lambda V_K, \mathsf{A}(X)), \Psi ) \cong \pi_0( \mathrm{map}^{(\id\otimes \psi, 1\otimes \Psi))} (\Lambda V_K, H(S^n) \otimes \mathsf{A}(X)) ),  
    \end{equation*}
    with the dga-homomorphism $(\id \otimes \psi,1\otimes \Psi) \colon \Lambda V_K \rightarrow \bigl( H(S^n) \otimes \mathsf{A}(A) \bigr) \oplus_{\bbK \otimes \mathsf{A}(A)} \bbK \otimes \mathsf{A}(X)$.
    %\textcolor{blue}{Moreover, the homomorphism is natural in $Y$.}
\end{cor}
We would like to remark that $\pi_0( \mathrm{map}^{(\id\otimes \psi, 1\otimes \Psi))} (\Lambda V_K, H(S^n) \otimes \mathsf{A}(X)) )$ does not carry a group structure a priori, but it is equipped with the group structure from $\pi_n( \mathrm{map}^\psi(\Lambda V_K, \mathsf{A}(X)), \Psi )$ under the bijection of Corollary \ref{cor: Corollary of relative adjunction}.
It is this groups we would like to understand in terms of derivation.

To this end, for each cdga $B$, we denote by $sB$ the \emph{shifted} chain complex with $(sB)^n = B^{n+1}$ and the `same' differential, i,e., $d^{sB} = d$.
The cdga $H(S^n) \otimes B$ decomposes additively into $H^0(S^n) \otimes B \oplus H^n(S^n) \otimes s^nB$.
Therefore, each grading-preserving linear map $F \colon C \rightarrow H(S^n) \otimes B$ with a chain complex $C$ as domain decomposes uniquely as follows:
\begin{equation}\label{eq: dec dga into derivation}
    F = 1 \otimes \theta^0_F + \mathrm{vol}_{S^n} \otimes \theta^n_F.
\end{equation}
An elementary calculation shows that $f \colon C \rightarrow H(S^n) \otimes B$ is a homomorphism of graded algebras if and only if $\theta^0_f$ is a cdga-homomorphism and $\theta^n_f$ is a $\theta^0_f$-derivation in the following sense.
\begin{definition}
    Let $\Phi \colon C \rightarrow B$ be a homomorphism of graded algebras and let $I \subseteq B$ be a differential ideal\footnote{We allow $I=B$ to be the entire algebra.}.
    A \emph{$\Phi$-derivation} of degree $n$ is a linear map $\theta \colon C \rightarrow I$ lowering the degree by $n$ and satisfying
    \begin{equation*}
        \theta(xy) = \theta(x)\Phi(y) + (-1)^{n|x|}\Phi(x)\theta(y)
    \end{equation*}
    Denote by $\Der^\Phi_n(C,I)$ the set of all $\Phi$-derivations of degree $n$.
    Together, they form the graded vector space
    \begin{equation*}
        \mathrm{Der}^\Phi(C,I) = \bigoplus_{n \in \mathbb{Z}} \mathrm{Der}_n^\Phi(C,I).
    \end{equation*}
    If $\Phi$ is a dga-homomorphism between differential graded algebras $C$ and $B$, then there is a differential $\delta \colon \mathrm{Der}^\Phi_n(C,I) \rightarrow \mathrm{Der}^\Phi_{n-1}(C,I)$ given by $\delta(\theta) = d\theta - (-1)^n \theta d$.
\end{definition}
\begin{lemma}\label{lem: homotopy group vs derivations}
    Let $C$ be a cofibrant cdga, $\Phi \colon C \rightarrow B$ a dga-homomorphism, $p \colon B \twoheadrightarrow T$ a surjective dga-homomorphism, and set $\varphi = p \circ \Phi$.
    Then, for all $n \geq 1$, there is a natural bijection
    \begin{equation}\label{eq: homomorphism - derivation decomposition}
        \pi_0\bigl( \mathrm{map}^{(\id\otimes\varphi,1\otimes\Phi)}(C,H(S^n) \otimes B) \bigr) \xrightarrow{\cong} H_n(\mathrm{Der}^\Phi\bigl( C, \ker \, p ),\delta \bigr); \qquad [F] \mapsto [\theta_F^n]
    \end{equation}
    %\textcolor{blue}{Moreover, if $n\geq 2$, then the map is a group homomorphism}. 
\end{lemma}
\begin{proof}
    An explicit computation shows that each graded linear map $F \colon C \rightarrow H(S^n) \otimes B$, which decomposes uniquely into $1 \otimes \theta^0_F + \vol_{S^n} \otimes \theta^n_F$, is an algebra homomorphism if and only if $\theta^0_F$ is a graded algebra homomorphism and $\theta^n_F$ is a $\theta^0_F$-derivation.

    Moreover, the graded linear map $F$ commutes with the differentials if and only if $\theta^0_F$ commutes with the differential and $\theta^n_F$ is $\delta$-closed, that is $\delta(\theta^n_F)=0$.

    By definition, $F \in \mathrm{map}_0^{(\id\otimes\varphi,1\otimes\Phi)}(C,H(S^n) \otimes B) \bigr)$ is a zero-simplex in this simplicial set, if and only if it is a cdga-homomorphism that makes the following diagram commute
    \begin{equation*}
        \xymatrix{ C \ar[rrr] \ar[rrrd]_{(\id \otimes \varphi,1\otimes \Phi)\qquad }  &&& H(S^n)\otimes B \ar[d]^{(\id \otimes p, \varepsilon \otimes \id)} \\
        &&& H(S^n)\otimes T \oplus_{\bbK \otimes T} \bbK \otimes B,}
    \end{equation*}
    where $\varepsilon \colon H(S^n) \rightarrow H^0(S^n) = \bbK$ is the augmentation.
    Since $\varepsilon(\vol_{S^n}) = 0$, we deduce that
    \begin{equation*}
        \theta^0_F = \Phi \qquad \text{ and } \qquad p \circ \theta^n_F = 0.
    \end{equation*}

    To show that the map in the statement is well defined and bijective, we will prove that homotopies are, in fact, in one-to-one correspondence with $\delta$-coboundaries.
    Abbreviate the cdga $\Lambda [ t,dt \, | \, d(t) = dt]$ to $I$.
    A homotopy $H$ between $F_0$ and $F_1$ over $T$ is a dga-homomorphism $H$ making the following diagram commute:
    \begin{equation*}
        \xymatrix{ C \ar[rr]^H \ar[rrd]_{ (1\otimes 1 \otimes \varphi, 1 \otimes 1 \otimes \Phi) \qquad \qquad} && H(S^n) \otimes I \otimes B \ar[d] \ar@<0.5ex>[rr]^{\mathrm{ev}_0} \ar@<-0.5ex>[rr]_{\mathrm{ev}_1} && H(S^n) \otimes B \ar[d] \\
        && H(S^n)\otimes (I \otimes T) \oplus_{\bbK \otimes I \otimes T} \bbK \otimes I \otimes B \ar@<0.5ex>[rr]^{\mathrm{ev}_0} \ar@<-0.5ex>[rr]_{\mathrm{ev}_1} &&  H(S^n)\otimes T \oplus_{\bbK \otimes T} \bbK \otimes B.  }
    \end{equation*}
    The homotopy can be uniquely expressed as 
    \begin{equation}\label{eq: Homotopy Decomposition App}
        \begin{split}
            H(v) =& 1 \otimes 1 \otimes \Phi(v) + 1 \otimes t \otimes h^0(v) +   \vol_{S^n} \otimes 1 \otimes \theta_{F_0}^n(v) \\
            &+ \vol_{S^n} \otimes t \otimes (\theta_{F_1}^n - \theta_{F_0}^n)(v) + \vol_{S^n} \otimes dt \otimes h^n(v).
        \end{split}  
    \end{equation}
    Arguing as before, commutativity of this diagram implies that 
    \begin{itemize}
        \item[(i)] $h^0(v) = 0$, which means that $H$ is a pointed homotopy,
        \item[(ii)] $ph^n(v) = 0$.
    \end{itemize}
    From 
    \begin{align*}
        \begin{split}
            dH(v) =& d\Bigl(1 \otimes 1 \otimes \Phi(v) + 1 \otimes t \otimes h^0(v) +   \vol_{S^n} \otimes 1 \otimes \theta_{F_0}^n(v) \\
            &+ \vol_{S^n} \otimes t \otimes (\theta_{F_1}^n - \theta_{F_0}^n)(v) + \vol_{S^n} \otimes dt \otimes h^n(v)\Bigr).
        \end{split} \\
        \begin{split}
            =& 1 \otimes 1 \otimes d\Phi(v) + (-1)^n\vol_{S^n} \otimes 1 \otimes d\theta_{F_0}^n(v) + (-1)^n\vol_{S^n} \otimes dt \otimes (\theta_{F_1}^n - \theta_{F_0}^n)(v)\\
            &+ (-1)^{n}\vol_{S^n} \otimes t \otimes d(\theta_{F_1}^n - \theta_{F_0}^n)(v) + (-1)^{n+1} \vol_{S^n} \otimes dt \otimes dh^n(v) 
        \end{split}
    \end{align*}
    and 
    \begin{align*}
        \begin{split}
        H(dv) =& 1 \otimes 1 \otimes \Phi(dv) + \vol_{S^n} \otimes 1 \otimes \theta^n_{F_0}(dv) + \vol_{S^n} \otimes t \otimes \bigl( \theta_{F_1} - \theta_{F_0} \bigr)(dv) \\
        &+ \vol_{S^n} \otimes dt \otimes  h^n(dv),
        \end{split}
    \end{align*}
    together with the fact that $\theta_{F_j}^n$ are derivations of degree $d$, we conclude that $H$ is a chain map if and only 
    \begin{equation*}
         (\theta_{F_1}^n - \theta_{F_0}^n)(v) = (-1)^n \bigl( h^n(dv) - (-1)^{n+1} dh^n(v) \bigr) = (-1) \bigl( dh^n(v) - (-1)^{n+1} h^n(dv) \bigr).   
    \end{equation*}
    From the observation above, we see that $v \mapsto dt \otimes h^n(v)$ is a derivation of degree $n$, and a straighforward calculation reveals that $v \mapsto h^n(v)$ must be therefore a derivation of degree $n+1$.
    Thus, $\Theta :=  h^n$ is a derivation of degree $n+1$ yielding a boundary between $\theta_{F_0}$ and $\theta_{F_1}$.

    Conversely, a given boundary $\Theta \in \Der_{n+1}^{\Phi}(C,\ker p)$ between $\theta^n_{F_0}$ and $\theta^n_{F_1}$ gives rise to a homotopy over $T$ by plugging in  $\Theta$ for $h^n$ in equation (\ref{eq: Homotopy Decomposition App}).
\end{proof}
With all the ingredients at hand, we can now formulate and proof the main theorem of this section.
%
% I will formulate after I have introduced the notaion.
\begin{theorem}\label{thm: hAut via derivations (general)}
    Let $\iota\colon A\subseteq X$ be the inclusion of a CW-pair with $X$ a finite CW-complex and let $Y$ be a nilpotent topological space.
    For a given map $f \colon A \rightarrow Y$, let $\mathrm{C}_f(X,Y)$ be the topological space of all continuous maps that restrict to $f$ on $A$ and let $F \colon X \rightarrow Y$ and extension of $f$.
    Let $\mathsf{A}(\iota)\colon \mathsf{A}(X) \twoheadrightarrow \mathsf{A}(A)$ be a surjective model for $\iota$, $(\Lambda V_Y,d)$ a Sullivan model for $\Omega(S(Y))$, the polynomial differential forms on the singular set $S(Y)$, and let $\varphi \colon \Lambda V_Y \rightarrow \mathsf{A}(A)$ be a rational model for $f$ that is covered by $\Phi \colon \Lambda V_Y \rightarrow \mathsf{A}(X)$, the rational model for $F$.% such that $\mathsf{A}(\iota) \circ \Phi = \varphi$.
   % Let further $\mathrm{Der}_n^\Phi(\Lambda V_Y,\ker \mathsf{A}(\iota))$ denote the space of all $\Phi$-derivations, 
    
    Then, for all $n\geq 1$, the map that assigns to a homotopy class $G \colon S^n \rightarrow \mathrm{C}_f(X,Y)$ the associated derivation $\theta^n_{G}$ of its rational model is a bijection of sets
    \begin{equation*}
        \pi_n(\mathrm{C}_f(X,Y),F) \otimes \bbK \longrightarrow H_n(\mathrm{Der}^\Phi(\Lambda V_Y,\ker \mathsf{A}(\iota)),\delta).
    \end{equation*}
    If $n\geq 2$ or if $n=1$ and $X=Y$ and $F=\id$, then the bijections are also group homomorphisms.
\end{theorem}
\begin{proof}
    It is enough to prove the theorem for $\bbK = \Q$ because $(-)\otimes_\Q\mathbb{K}$ is an exact functor and the right-hand side is compatible with tensor products in the sense that $H_n\bigl(\Der^\Phi(\Lambda V_X,\ker \mathsf{A}(\iota))\bigr) \otimes \bbK \cong  H_n\bigl(\Der^{\Phi\otimes\id_\bbK}(\Lambda V_X\otimes \bbK,\ker [\mathsf{A}(\iota) \otimes \id_\bbK] ) \bigr) $.

    Recall the zig-zag of weak equivalences of relative mapping spaces (\ref{eq: zig-zag of relative topolgical mapping spaces}).
    It can be further composed the `homotopy-unit' $\eta \colon K \rightarrow  \langle \Lambda V_K\rangle$, meaning, the lift of the unit map $K \rightarrow \langle \Omega(K)\rangle$ under the quasi-isomorphism $\Lambda V_K \rightarrow \Omega(K)$, which yields the zig-zag
    \begin{equation*}
         \mathrm{C}_f(X,Y) \xleftarrow{\simeq} |SC_f(X,Y)| \xrightarrow{\simeq} | \mathrm{map}_f(S(X),S(Y))| \xrightarrow{\eta\circ (\placeholder)} |\mathrm{map}_{\eta\circ f}(S(X),\langle\Lambda V_{Y} \rangle))|.
    \end{equation*}
    By applying homotopy groups and inverting the isomorphisms induced by the left horizontal map together with Corollary \ref{cor: Corollary of relative adjunction} and Lemma \ref{lem: homotopy group vs derivations}, we obtain a homomorphism
    \begin{align*}
        \pi_n( \mathrm{C}_f(X,Y),F) \rightarrow  \pi_n(\mathrm{map}_{\eta \circ f}(S(X),\langle \Lambda V_{Y}\rangle),\eta \circ F) &\cong \pi_n(\mathrm{map}^{\eta \circ f}(\Lambda V_{Y}, \Omega( S(X)) ), \eta\circ F) \\ 
        &\cong \pi_n(\mathrm{map}^\varphi(\Lambda V_{Y}, \mathsf{A}(X) ), \Phi), 
    \end{align*}
    where $\Phi$ is a model for $\eta \circ F$ and $\varphi = \mathsf{A}(\iota)\circ \Phi$ is a model for $\eta \circ f$.
    
    It is proved in \cite[Section 11.2]{Bousfield1976PLdeRham} that the `homotopy-unit' $\eta \colon K \rightarrow \langle \Lambda V_K\rangle$ is a rationalisation if $K$ is a connected, nilpotent, simplicial set of finite $\Q$-type.
    Since $X$ is a finite nilpotent CW-complex, it follows from the work of Hilton-Mislin-Roitberg \cite{Hilton1975Localisation} that postcomposition with this map yields a $\Q$-localisation $\eta\circ (\placeholder) \colon \mathrm{map}_f(S(X),S(Y)) \rightarrow \mathrm{map}_{\eta\circ f}(S(X),\langle\Lambda V_{Y} \rangle))$. 
    In particular, the above homomorphism $\pi_n( \mathrm{C}_f(X,Y),F) \rightarrow \pi_n(\mathrm{map}^\varphi(\Lambda V_{Y}, \mathsf{A}(X) ), \Phi) $ is, in fact, an isomorphism after completing the domain with $\Q$.   

    By Lemma \ref{lem: homotopy group vs derivations}, we obtain the claimed bijection between the groups $\pi_n(\mathrm{C}_f(X,Y);F)\otimes \Q$ and $H_n(\Der^\Phi(\Lambda V_{Y},\ker \mathsf{A}(\iota))$.

    \vspace{6pt}
    
    It remains to show that this bijection is, in fact, a homomorphism if either $n\geq 2$ or $n=1$ and $\mathrm{C}_f(X,Y) = \hAut_A(X)$ and $F=\id$.
    We begin with the higher homotopy groups.
    By \cite[Theorem 2.2]{Hilton1975Localisation}, the fibre of a fibration is nilpotent if the total space is.
    Thus, we have an isomorphism $\pi_n(\mathrm{C}(X,Y)) \otimes \Q \cong \pi_n(\mathrm{C}(X,Y_\Q)$.
    The sum of $[g_1]$ and $[g_2]$ in $\pi_n(\mathrm{C}_f(X,Y),F) \otimes \Q$ is therefore represented by the composition
    \begin{equation*}
        \xymatrix{ S^n \ar[rr]^c && S^n \vee S^n \ar[rr]^{g_1 \vee g_2} && \mathrm{C}_f(X,Y_\Q) }
    \end{equation*}
    with the first map being the collapse map.
    Its adjoint under the exponential law for mapping spaces yields the composition
    \begin{equation*}
        \xymatrix{ S^n \times X \ar[rr]^-{c \times \id_X} && (S^n \vee S^n) \times X \ar[rr]^-{\mathrm{Ad}(g_1 \vee g_2)} && Y_\Q }
    \end{equation*}
    and because of the universal property of localisations, we get a factorisation
    \begin{equation*}
        \xymatrix{ S^n_\Q \times X_\Q \ar[rr]^-{c_\Q \times \id} && (S^n \vee S^n)_\Q \times X_\Q \ar[rr]^-{\mathrm{Ad}(g_1 \vee g_2)_\Q} && Y_\Q }
    \end{equation*}
    Since all involved spaces are nilpotent (it is here that we are using $n\geq 2$ to guarantee that $(S^n\vee S^n)_\Q = \langle \mathsf{M}_{S^n \vee S^n}\rangle$ ), 
    %If $N=1$
    the homotopy class of this map has a unique (up to homotopy) algebraic counterpart
    \begin{equation*}
        \xymatrix{ H(S^n) \otimes \mathsf{A}(X) && \mathsf{M}_{S^n \vee S^n} \otimes \Lambda V_X \ar[ll]_{c^\bullet \otimes \id } \ar[d]_\simeq  && \Lambda V_Y \ar[ll]_-{(g_1\vee g_2)^\bullet} \ar[lld]^{(f\vee g)^\bullet}\\
        (H(S^n) \oplus_{\Q} H(S^n))\otimes \mathsf{A}(X) \ar[u]^+ && H(S^n\vee S^n) \otimes \mathsf{A}(X) \ar@{=}[ll] \ar[llu]^{H(c)\otimes \id} && }
    \end{equation*}
    where $\mathsf{M}_{S^n\vee S^n}$ is the minimal model for $S^n \vee S^n$ and $c^\bullet \colon \mathsf{M}_{S^n\vee S^n} \rightarrow H(S^n)$ is a rational model for the collapse map $c$ and $(g_1\vee g_2)^\bullet$ is a rational model for $g_1 \vee g_2$.
    The truncated chain complex $\mathsf{M}_{S^n\vee S^n}^{\leq n}$ is isomorphic to $H(S^n\vee S^n) = H(S^n) \oplus_\Q H(S^n)$, so the map $c^\bullet$ agrees with the sum.

    On the other, the dga-homomorphism $(f \vee g)^\bullet \colon \Lambda V_Y \rightarrow \bigl( H(S^n) \oplus_\Q H(S^n) \bigr) \otimes \mathsf{A}(X)$ decomposes uniquely into $1 \otimes \Phi + \vol_{S^n_1} \otimes \theta^n_f + \vol_{S^n} \otimes \theta^n_g$ from which we deduce 
    \begin{equation*}
        \theta^n_{[f] + [g]} = \theta^n_{(f\vee g) \circ c} = \theta_f^n + \theta_g^n. 
    \end{equation*}

    For $n=1$ we cannot apply the same strategy because $S^1 \vee S^1$ is not nilpotent and therefore the composition $S^1 \vee S^1 \rightarrow \langle \mathsf{M}_{S^1\vee S^1}\rangle$ does not need to yield a bijection $[S^1 \vee S^1,Y_\Q] \cong [\langle \mathsf{M}_{S^1\vee S^1}\rangle, Y_\Q]$, which this is the reason why we restrict our consideration to the case $X = Y$ and $F = \id$.
    Indeed, in this case, $\hAut_A(X)$ carries a monoid structure given by composition.
    We first replace $\mathsf{A}(\iota) \colon \mathsf{A}(X) \twoheadrightarrow \mathsf{A}(A)$ by a strict Sullivan model $\iota^\bullet \colon \Lambda V_X \twoheadrightarrow \Lambda V_A$ in the sense that the follwing diagram, in which the horizontal arrows are quasi-isomorphisms, (strictly) commute
    \begin{equation*}
        \xymatrix@R-0.5em{ \Lambda V_X \ar[rr]^\simeq \ar@{->>}[d]_{\iota^\bullet} && \mathsf{A}(X) \ar@{->>}[d]^{\mathsf{A}(\iota)} \\ \Lambda V_A \ar[rr]^\simeq && \mathsf{A}(A). }
    \end{equation*}
    If $\psi$ and $\Psi$ denote a homotopy lift of $\varphi$ and $\Phi$, respectively, then the strict model induces an isomorphism $\pi_n(\mathrm{map}^{\psi}(\Lambda V_Y,\Lambda V_X),\Psi) \xrightarrow{\cong} \pi_n(\mathrm{map}^\varphi(\Lambda V_Y,\mathsf{A}(X)),\Phi)$ by Lemma \ref{lem: quasi-iso invariant fibres}.
    Thus, $H_n(\Der^\Psi(\Lambda V_Y,\ker \iota^\bullet),\delta)$ and $H_n(\Der^\Phi(\Lambda V_X,\ker \iota^\bullet),\delta)$ are isomorphic if they are equipped with the groups structure coming from the homotopy groups.
    
    Theorem \ref{thm: fundamental theorem} implies that the groups remain isomorphic under pre-composition with quasi-isomorphisms between Sullivan models $\Lambda V_Y \rightarrow \Lambda W_Y$ of $Y$, so we are allowed to pick the same Sullivan model in domain and target because we assumed $X=Y$.
    Since $\Psi$ is then a rational models of the identity on $X$, we may assume that $\Psi = \id_{\Lambda V_X}$.
    
    We will show that the group structure on $\pi_1(\hAut_A(X),\id)$ that comes from the monoid structure agrees with the addition on $\Der^{\id}_1(\Lambda V_X,\ker \iota^\bullet)$.
    This will prove the second statement, because this groups structure agrees with the fundamental group structure by the Eckmann-Hilton argument.

    Under the exponential law, the composition of two maps $g_1,g_2 \colon S^1 \rightarrow \hAut_A(X)$ corresponds to the composition of the maps in the upper horizontal line of the following diagram:
    \begin{equation*}
        \xymatrix@R-0.5em{ S^1  \times X \ar[rr]^{\Delta \times \id_X} && S^1 \times S^1 \times X \ar[rr]^-{\id_{S^1}\times \mathrm{Ad}(g_2)} && S^1 \times X \ar[rr]^{\mathrm{Ad}(g_1)} && X \\
        S^1  \times A \ar[rr]^{\Delta \times \id_A} \ar@{^{(}->}[u] && S^1 \times S^1 \times A \ar[rr]^{\id_{S^1}\times \mathrm{pr}_A} \ar@{^{(}->}[u] && S^1 \times A \ar[rr]^{\mathrm{pr}_A} \ar@{^{(}->}[u] && A \ar@{^{(}->}[u].   }
    \end{equation*}
    Since all involved spaces are nilpotent, each continuous map in this composition can be modelled by a dga-homomorphism
    \begin{equation*}
        \xymatrix{ H^1(S^1) \otimes \Lambda V_X && H^1(S^1) \otimes H^1(S^1) \otimes \Lambda V_X \ar[ll]_-{H(\Delta)\otimes \id}^-{=\cup \otimes \id} && H(S^1) \otimes \Lambda V_X \ar[ll]_-{\id \otimes \mathrm{Ad}(g_2)} & \Lambda V_X \ar[l]_-{\mathrm{Ad}(g_1)}. }
    \end{equation*}

    The natural decomposition (\ref{eq: dec dga into derivation}) together with the fact that $\theta^0_{g_1} = \theta^0_{g_2} = \id$ (because $F= \id$) implies
    \begin{equation*}
        \bigl(\id \otimes \mathrm{Ad}(g_2)\bigr) \circ \mathrm{Ad}(g_1) = 1 \otimes 1 \otimes \id_{\Lambda V_X} + \vol \otimes 1 \otimes \theta^1_{g_1} + 1 \otimes \vol \otimes \theta^1_{g_2} + \vol \otimes \vol \otimes \theta^1_{g_2}\otimes^1_{g_1}.
    \end{equation*}
    From $\vol  \cup \vol = 0 \in H^2(S^1)$, we conclude that the composition with $H(\Delta) \otimes \id_{\Lambda V_X}$ yields
    \begin{equation*}
        (H(\Delta)\otimes \id) \circ  \mathrm{Ad}(g_2) \circ \mathrm{Ad}(g_1) = 1 \otimes \id + \vol \otimes (\theta^1_{g_1} + \theta^1_{g_2}). 
    \end{equation*}
    In conclusion, the group structure on $\pi_1(\hAut_A(X))$ corresponds under the bijection of Corollary \ref{cor: Corollary of relative adjunction} to the addition on $H_1(\Der^{\id}(\Lambda V_X,\ker \iota^\bullet),\delta)$ as claimed. 
\end{proof}

\section{Real Homotopy Theory of Orbifold Resolutions}\label{Section: RHT Orbifold Resolutions}

This section is devoted to the proof of the main topological result of this article, namely Theorem \ref{Main thm: Orbifold Robustness}, which essentially says that $\pi_2(B\Diff(M))$ contains a free abelian subgroup if $M$ is obtained from an orbifold by blowing-up singularities.

\subsection{Topological Properties of Blow-Up Families}\label{Subsection: Top Properties Blow Up Families}

We begin introducing the twisted families of blow-ups that we use to resolve the singularities with.
Recall that the disc bundles associated to the complex line bundle $\mathcal{O}(k) \rightarrow \CP^1$ can be obtained from the Hopf fibration using the Borel construction:
\begin{equation*}
    D\mathcal{O}(k) = S^3 \times_{S^1, (\placeholder)^{k}} D^2 = (S^3 \times D^2)/\sim \qquad (p,\lambda) \sim (\mathrm{e}^{\iu\theta}p,\mathrm{e}^{\iu k\theta}\lambda).
\end{equation*}
We can use the Hopf fibration once more to construct a $D\mathcal{O}(k)$-bundle over $S^2$ by setting
\begin{equation*}
    \mathcal{DO}(k) = S^3 \times_{S^1,(\placeholder)^1} D\mathcal{O}(k) =  \bigl(S^3 \times D\mathcal{O}(k)\bigr) / \sim  \qquad (q,[p,\lambda]) \sim (\mathrm{e}^{\iu\tau}q,[\mathrm{e}^{\iu\tau}p,\lambda]).
\end{equation*}

Since $\mathcal{DO}(k)$ is obtained from the Borel construction of the Hopf fibration, the classifying map of this fibre bundle is given by the composition
\begin{equation*}
    f_{\mathcal{DO}(k)}  \colon S^2 \rightarrow BS^1 = \CP^\infty \xrightarrow{B\twist} B\Diff(D\mathcal{O}(k))
\end{equation*}
where $\twist \colon S^1 \rightarrow \Diff(D\mathcal{O}(k))$ is the obvious inclusion homomorphism
\begin{equation*}
        \twist \colon S^1 \rightarrow \Diff(D\mathcal{O}(k)), \qquad \mathrm{e}^{\iu\tau} \mapsto \bigl( [p,\lambda] \mapsto [\mathrm{e}^{\iu \tau}p,\lambda] \bigr).
\end{equation*}
We would like to compute the order of the elements $[f_{\mathcal{DO}(k)}] \in \pi_2(B\Diff(D\mathcal{O}(k)))$.
To this end, for a closed cofibration $A \hookrightarrow X$, let $\mathrm{hAut}(X,A)$ denote the monoid of all homotopy automorphism of pairs. 
Since diffeomorphisms on manifolds with boundary preserve the boundary by default, we get an obvious monoid-homomorphism $\Diff(D\mathcal{O}(k)) \rightarrow \hAut(D\mathcal{O}(k),\partial D\mathcal{O}(k))$.

\vspace{6pt}

\begin{prop}\label{prop: classifying maps and their lifts}
 $ $
 $ $
    \begin{itemize}
        \item[1.)] The classifying map $f_{\mathcal{DO}(k)}$ lifts to a map $S^2 \rightarrow B\Diff_\partial(D\mathcal{O}(k))$ if $1 \leq k \leq 2$.
        \item[2.)] The image of $[f_{\mathcal{DO}(k)}]$ in $\pi_2( B\hAut(D\mathcal{O}(k),S^3/\Z_k) )$ has infinite order.
    \end{itemize}
\end{prop}
\begin{proof}
    We start with the proof of the first statement, so let us assume that $k \in \{\pm 1,\pm 2\}$.
    Since $\mathcal{DO}(k) \rightarrow S^2$ is a fibre bundle over a sphere, it can be obtained by a clutching construction, which means it isomorphic as a fibre bundle to the following pushout
    \begin{equation*}
        \xymatrix{ S^1 \times D \mathcal{O}(k) \ar@{^{(}->}[rr]^{(\incl,\twist)} \ar@{^{(}->}[d] && D^2 \times D \mathcal{O}(k) \ar[d] \\
        D^2 \times D \mathcal{O}(k) \ar[rr] && \mathcal{DO}(k). }
    \end{equation*}
    Thus, it suffices to show that the map $\mathsf{tw} \colon S^1 \rightarrow \mathrm{Diff}(D\mathcal{O}(k))$ can deformed to a map with values in $\Diff_\partial(D\mathcal{O}(k))$.

    To this end, we consider the following commutative diagram
   \begin{equation*}\label{eq: isometry vs homotopy auto fibration}
       \xymatrix{ \Omega \CP^1 \ar[r]^\simeq \ar@{-->}[d] & \mathrm{hofib}_{1}(\mathrm{incl}) \ar[r] \ar@{..>}[dl]_\Phi \ar[d] & S^1 \ar[r] \ar[d]^{\mathsf{tw}} & S^3 \ar[d]^\ell \\
       \Diff_{\partial} (D\mathcal{O}(k)) \ar[r]_\simeq \ar[d] & \mathrm{hofib}_{\id}(\restriction_{S^3/\Z_k}) \ar[r] \ar[d] & \Diff(D\mathcal{O}(k),S^3/\Z_k) \ar[r]^-{\restriction_{S^3/\Z_k}} & \Diff(S^3/\Z_k) \ar[d] \\
       \hAut_{\partial} (D\mathcal{O}(k)) \ar[r]_\simeq & \mathrm{hofib}_{\id}(\restriction_{S^3/\Z_k}) \ar[r] & \hAut(D\mathcal{O}(k),S^3/\Z_k) \ar[r]^-{\restriction_{S^3/\Z_k}} & \hAut(S^3/\Z_k),}
   \end{equation*}
   where $\ell \colon S^3 \curvearrowright S^3/\Z_k$ is action given by left multiplication\footnote{It is here, where we use the assumption $1 \leq |k| \leq 2$ because only in this case is the normaliser subgroup of $\Z_k$ inside $S^3$ the group $S^3$.}. 
   The dotted arrow is the map $\Phi$ defined as follows:
   If we use the concrete description $\mathrm{hofib}_1(\mathrm{incl}) = \{ (\mathrm{e}^{\iu\tau},\gamma) \in S^1 \times \mathrm{C}([0,1],S^3) \, : \, \gamma(0) = \mathrm{e}^{\iu\tau}, \gamma(1) = 1  \}$, then %$\Phi$ that is induced by a null-homotopy $H\colon [0,1] \times S^1  \rightarrow S^3$ with $H(\placeholder,0) = \mathrm{incl}$ and $H(\placeholder, 1) = 1$ as follows:
   we define $\Phi(e^{\iu\tau},\gamma)$ to be the following diffeomorphism:
   \begin{equation*}
       D\mathcal{O}(k) \ni [p,\lambda] \mapsto \begin{cases}
           [ e^{\iu \tau}p, 0 ], & \text{if } \lambda = 0, \\
           [ \gamma(|\lambda|) \cdot  |\lambda| \cdot (\bar{\lambda}/|\lambda|)^{1/k}\cdot  p , 1 ], & \text{if } \lambda \neq 0. 
       \end{cases} 
   \end{equation*}

   A choice of a null-homotopy $H \colon [0,1] \times S^1 \rightarrow S^3$ with $H(\placeholder,0)=\mathrm{incl}$ and $H(\placeholder,1) = 1$ defines a map $S^1 \rightarrow \mathrm{hofib}_{1}(\incl)$, which is then the desired lift.

   \vspace{6pt}
   
   To prove the second statement, we first restrict to the special case where $k = \pm 1$.
   We wish to show that
   \begin{equation*}
       \pi_1(S^1,1)\otimes \R \xrightarrow{\pi_1(\mathsf{tw} \otimes \R)} \pi_1(\hAut(D\mathcal{O}(\pm 1),S^3),\id)\otimes \R .
   \end{equation*}
   is injective.
   To this end, observe that the map $\twist$ can be extended by (the adjoint of) the group action of $\mathrm{U}(2)$ on $D\mathcal{O}(\pm 1)$. 
   Together with the tautological action of $\mathrm{PU}(3)$ on $\CP^2$ they fit into the following commutative diagram
   \begin{equation*}\label{eq: linear action on disc bundles}
       \xymatrix{ S^1 \times D\mathcal{O}(\pm 1) \ar[rr] \ar[d]_\twist && \mathrm{U}(2) \times D\mathcal{O}(\pm 1) \ar[rr] \ar[d] &&  D\mathcal{O}(\pm 1) \ar[d] \\
        S^1 \times \CP^2 \ar[rr] && \mathrm{PU}(3)\times \CP^2 \ar[rr] && \CP^2, }
   \end{equation*}
   where the inclusion $S^1 \cdot \mathrm{SU}(2) = \mathrm{U}(2) \rightarrow \mathrm{PU}(3)$ is given by $A \mapsto \mathrm{diag}(A,1)$ and the map $D\mathcal{O}(1) \rightarrow D\mathcal{O}(1)/S^3 = \mathrm{Th}(\mathcal{O}(1)) = \CP^2$ is the collapse map.
   For $D\mathcal{O}(-1)$, we use the map $D\mathcal{O}(-1) \xrightarrow{ \id \times \bar{\cdot} } D\mathcal{O}(1) \rightarrow \CP^2$ instead.

   Overall, we now end up with the following commutative diagram
   \begin{equation*}\label{eq: action comparison}
       \xymatrix{ &\Omega \CP^2 \ar[d] \ar@{=}[rr] && \Omega \CP^2 \ar[d] \\
       S^1 \ar[r]  & \mathrm{U}(2) \ar[r] \ar[d] & \hAut(D\mathcal{O}(\pm 1),S^3) \ar[r]^-{(\placeholder)/S^3}  &\hAut_\ast(\CP^2) \ar[d] \\ 
       & \mathrm{PU}(3) \ar[rr] && \hAut(\CP^2), }
   \end{equation*}
   By a result of Sasao \cite{Sasao1974CP}, the map $\mathrm{PU}(3) \rightarrow \hAut(\CP^2)$ is a real homotopy equivalence.
   It follows that comparison map $\mathrm{U}(2) \rightarrow \hAut_\ast(\CP^2)$ is a real homotopy equivalence, which implies that $0 \neq [\twist] \in \pi_1( \hAut(D\mathcal{O}(\pm1),S^3) )\otimes \R$ as claimed.

   \vspace{6pt}

   To generalise this result to all disc-bundles $D\mathcal{O}(k)$ (excluding $k=0$), we observe that we have a $\Z_k$-action on $D\mathcal{O}(\pm 1)$ that is given by $\zeta_\bullet[p,\lambda] = [p,\zeta\cdot \lambda]$, which agrees with the restriction of the twist action $\mathsf{tw}\colon S^1 \curvearrowright D\mathcal{O}(\pm1)$ to $\Z_k$.
   Like the twist action, it can be extended to an action on $\mathrm{Th}(D\mathcal{O}(\pm 1)) = \CP^2$ so that $\mathrm{Th}(\mathcal{O}(\pm 1))/\Z_k = \mathrm{Th}(\mathcal{O}(k))$.
   The quotient map $p_k \colon \mathrm{Th}(\mathcal{O}(\pm 1)) \rightarrow \mathrm{Th}(\mathcal{O}(k)) = \mathrm{Th}(\mathcal{O}(\pm 1))/\Z_k$ induces an isomorphism between all real homotopy groups because the involved space are simply connected and quotient maps of finite group actions induce isomorphisms on real homology \cite[Theorem III.2.4]{Bredon1972IntroCptTrafoGrps}.
   Thus, the (extended) twist actions fit into a commutative diagram 
   \begin{equation*}
       \xymatrix{ S^1 \times \mathrm{Th}(\mathcal{O}(\pm 1)) \ar[rr]^{\mathsf{tw}/S^3} \ar[d] && \mathrm{Th}(\mathcal{O}(\pm 1)) \ar[d] \\
       S^1 \times \mathrm{Th}(\mathcal{O}(k) \ar[rr]^{\mathsf{tw}/(S^3/\Z_k)} && \mathrm{Th}(\mathcal{O}(k)) }
   \end{equation*}
   with the vertical arrows inducing isomorphisms on real homotopy groups.
   Hence, if the (adjoint of the) lower horizontal map $\mathsf{tw}/(S^{3}/\Z_k) \colon (S^1,\{1\}) \rightarrow (\hAut_\ast(\mathrm{Th}(\mathcal{O}(k))),\id)\otimes \R$ would present the zero element in $\pi_1(\hAut(\mathrm{Th}(\mathcal{O}(k))), \id) \otimes \R$, then $\mathsf{tw}/S^3$ would also represent the zero element, and we showed above that this is not the case.
\end{proof}
Abusing notation, we denote the lift of the twist map given in the proof of the first part of Proposition \ref{prop: classifying maps and their lifts} again by $\mathsf{tw} \colon S^1 \rightarrow \Diff_\partial(D\mathcal{O}(k))$ as well as its adjoint map $\mathsf{tw} \colon S^1 \times D\mathcal{O}(k) \rightarrow D\mathcal{O}(k)$ if $1 \leq k \leq 2$.
This is legitimate because a null-homotopy $H\colon S^1 \times [0,1] \rightarrow S^3$ of the inclusion $S^1 \hookrightarrow S^3$ is unique up to homotopy.

From Example \ref{exmpl: hAut of disc bundles} we immediately obtain the next consequence.
\begin{cor}\label{cor: derivation of the twist}
    If $1 \leq |k| \leq 2$, then there is a $\xi \neq 0 \in \R$ such that
    \begin{equation*}
        [\theta^1_{\mathsf{tw}}] = \xi \cdot [a_2 \otimes b_3^\vee] \in H_1(\Der(\mathsf{M}_{D\mathcal{O}(k)},\ker \mathsf{M}(\iota)),\delta) \cong \pi_1( \hAut_\partial( D\mathcal{O}(k) ),\id ) \otimes \R.
    \end{equation*}
\end{cor}

\subsection{Family Resolutions of tailor-made orbifolds}\label{subsection: resolving orbifolds}

With the bundles $\mathcal{DO}(k) \rightarrow S^2$ at our disposal, we are now able to resolve singularities of orbifolds with tubular neighbourhoods of the form $S\times D^4/\Z_2$ in a `twisted fashion' by replacing them with $S \times \mathcal{DO}(k)$ to obtain fibre bundles over spheres.
In this subsection, we will classify this construction through continuous maps between the classifying space $B\Diff_\partial(D\mathcal{O}(k)) \rightarrow B\Diff(M)$ and study their effects on rational homotopy groups through appropriate rational models.

From now on, we will only consider orbifolds that are smooth, closed and simply connected.
Recall from the introduction that we call a pair $(X,\mathcal{S})$ consisting of an orbifold $X$ and closed smooth manifold $\mathcal{S} \subseteq X$ a \emph{tailor-made orbifold} if $\mathcal{S}$ satisfies the following properties:
   \begin{itemize}
       \item[(i)] Each path component $S \subseteq \mathcal{S}$ has a tubular neighbourhood $\mathrm{Tub}(S)$ of the form $S \times D^4/\Z_{k_S}$ with $k_S \in \{1,2\}$ and the antipodal action $\Z_2 \curvearrowright D^4$.
       \item[(ii)] $\mathcal{S}$ contains all singular points of $X$.
   \end{itemize}
Although one may think of $\mathcal{S}$ as the singular set of the orbifold, we wish to remark that $\mathcal{S}$ is allowed to contain regular points as well. 

Since $\mathcal{S}$ is closed, it has only finitely path components, so we may and do assume that the tubular neighbourhoods of different path components do not intersect. 
We denote the union of these tubular neighborhoods by $\mathrm{Tub}(\mathcal{S})$.
Furthermore, since each path component has a trivial normal bundle, $\mathcal{S}$ is an orientable manifold. 

Condition (ii) implies that the complement $U = X \setminus \Tub(\mathcal{S})$ is a smooth manifold. 
Obviously, as a topological space $X$ is the following pushout
\begin{equation*}
    \xymatrix{ \bigsqcup_{S \in \pi_0(\mathcal{S})}  S\times S^3/\Z_{k_S} \ar@{=}[r] & \partial U = \partial \mathrm{Tub}(\mathcal{S}) \ar[rr] \ar[d] && \mathrm{Tub}(\mathcal{S}) \ar@{=}[r] \ar[d] & \bigsqcup_{S \in \pi_0(\mathcal{S})} S \times D^4 /\Z_{k_S}  \\
    & U \ar[rr] && X, }
\end{equation*}
with $k_S \in \{\pm1,\pm2\}$ of course.
A resolving manifold $M$ is now described as the pushout
\begin{equation*}
    \xymatrix{ \bigsqcup_{S \in \pi_0(\mathcal{S})} S \times S^3 /\Z_{k_S} \ar[rr] \ar[d] &&  \bigsqcup_{S \in \pi_0(\mathcal{S})} S \times D\mathcal{O}(k_S)  \ar[d] \\
    U \ar[rr] && M, }
\end{equation*}
with $k_S \in \{\pm 1, \pm2\}$ accordingly to the type of tubular neighbourhood it resolves.
By the pushout property, the resolutions maps $\blowdown \colon D\mathcal{O}(k) \rightarrow D^4/\Z_{k}$ yield a resolution map $\blowdown \colon M \rightarrow X$.

\begin{lemma}
    The resolution map $\blowdown \colon M \rightarrow X$ induces an isomorphism on the fundamental group.
    In particular, $M$ is simply connected.
\end{lemma}
\begin{proof}
    Of course, we can obtain $M$ from $X$ by resolving one component at a time yielding a finite sequence of orbifolds starting from $X$ and ending with $M$.
    Since the resolution maps $\blowdown\colon \mathcal{O}(k) \rightarrow D^4/\Z_{k}$ induce isomorphisms on the fundamental group, an inductive application of van Kampens theorem implies that the each resolution map in the sequence induces an isomorphism on the fundamental group.
    Since $X$ is simply connected, so must be $M$. 
\end{proof}

Since the inclusion $\partial U \hookrightarrow U$ induces a surjection on differential forms $\Omega(U) \twoheadrightarrow \Omega(\partial U)$, the pushout description together with Example \ref{exmpl: pushout model} allows us write down a real model of $M$ in terms of its building blocks
\begin{equation}\label{eq: pushout model}
    \xymatrix{ \Omega( \partial U ) && \bigoplus_{S \in \pi_0(\mathcal{S})} \mathsf{M}_{S^2} \otimes \Omega(S) \ar[ll] \\
    \Omega(U) \ar@{->>}[u] && \mathsf{A}(M) \ar[ll] \ar@{->>}[u], }
\end{equation}
so that
\begin{equation*}
    \mathsf{A}(M) = \Omega(U) \oplus_{\Omega(\partial U)} \bigoplus_{S \in \pi_0(\mathcal{S})} \mathsf{M}_{S^2} \otimes \Omega(S).
\end{equation*}
Similarly, there is a pushout model for the orbifold given by
\begin{equation*}
    \mathsf{A}(X) = \Omega(U) \oplus_{\Omega(\partial U)} \bigoplus_{S \in \pi_0(\mathcal{S})} \Lambda V_{D^4} \otimes \Omega(S)
\end{equation*}
with $V_{D^4} = \mathrm{span}\{b_3, \vol_{D^4}\}$ and $db_3 = \vol_{D^4}$.
For later purpose, we observe that the two cdgas contain a common differential, graded subalgebra 
\begin{equation*}
    \mathsf{B}(M) = \mathsf{B}(X) = \Omega(U) \oplus_{\Omega(\partial U)} \bigoplus_{S \in \pi_0(\mathcal{S})} \R\cdot 1 \otimes \Omega(S).
\end{equation*}
%with $\R\cdot 1 \subseteq \mathsf{M}_{S^2}$ of course.

Recall that $\mathsf{M}_{S^2} = \Lambda[a_2,b_3\,|\,db_3 = a_2^2]$. 
We index the generators of the summand $\mathsf{M}_{S^2}\otimes 1$ inside $\mathsf{A}(M)$ belonging to $S$ by $a_{2,S}$ and $b_{3,S}$.
Observe that system $\{a_{2,S},b_{3,S} \, : \, S \in \pi_0(\mathcal{S})\}$ is linearly independent in the chain complex $\mathsf{A}(M)/\mathsf{B}(M)$.
\begin{lemma}\label{lem: Injection on Cohomology}
    The resolution map $\blowdown \colon M \rightarrow X$ induces an injection on real cohomology.
    Moreover, we have $H^2(M) = H^2(X) \oplus \mathrm{span}\{a_{2,S} \, : \, S \in \pi_0(\mathcal{S})\}$.
\end{lemma}
\begin{proof}
    The resolution map $D\mathcal{O}(k) \rightarrow D^4 /\Z_{k}$ restricts to the identity on the boundary, so it is modelled by the dga-homomorphism 
    \begin{equation*}
       \sigma\colon  \Lambda[ b_3, \vol_{D^4}] \rightarrow \Lambda [a_2,b_3]  = \mathsf{M}_{S^2} \quad \text{ induced by } \quad b_3 \mapsto  b_3/k, \ \vol_{D^4} \mapsto a_2^2/k.
    \end{equation*}
    This dga-homomorphism has a retract in the category of chain complexes: 
    \begin{equation*}
        \eta \colon \mathsf{M}_{S^2}  =  \Lambda[a_2,b_3] \rightarrow \Lambda[b_3,\vol_{D^4}] \quad \text{ given by } \quad a_2^{2n+1}b_3^m \mapsto 0 \ \text{ and } \ a_2^{2n}b_3^m \mapsto k^{n+m}\cdot \vol_{D4^n} b_3^m.
    \end{equation*}
    Since $\eta$ is the identity on the subalgebras generated by $b_3$ it extends to a retract of the real model of $ \mathsf{A}(\blowdown) = \id_{\Omega(U)} \oplus \bigoplus_{S \in \pi_0(\mathcal{S})}\sigma$.
    (Of course, $\eta$ and its extension are not dga-homomorphisms), but since cohomology is functorial with respect to chain maps, we obtain a retract of $H(\blowdown)$ on cohomology.
    Hence $H(\blowdown)$ is injective. 

    In degree $2$, the cokernel of the retract $\id_{\Omega(U)} \oplus \bigoplus_{S \in \pi_0(\mathcal{S})} \eta$ is the direct sum of all $\mathsf{M}^2_{S^2} \otimes 1$, which is spanned by all $a_{2,S}$.
    Clearly, the elements $(0,a_{2,S} \otimes 1)$ lie in the kernel of the differential of $\mathsf{A}(M)$ and they form a linearly independent system there.
    For degree reasons, each element in $\mathsf{A}^1(M)$ is the form $(\varphi, 1 \otimes \bm{\psi})$ with $\varphi \in \Omega^1(U)$ and $\bm{\psi} = \{\psi_{S}\}$ is a tuple of differential forms with $\psi_S \in \Omega^1(S)$.
    Since the differential acts componentwise, we see that the image of $d \colon \mathsf{A}^1(M) \rightarrow \mathsf{A}^2(M)$ and the linear hull of $\{(0,a_{2,S} \otimes 1) \, : \, S \in \pi_0(S)\}$ intersect in a trivial fashion.
    Thus, they are also a linear independent system in cohomology. 
\end{proof}
Since $M$ is simply connected, we can apply the `algorithm' to compute the minimal model of a simply connected topological space described in \cite[p.144 ff]{Felix2001RHT}, which yields the following result.
\begin{lemma}\label{lemma: minimal model in low degrees}
    If $\mathsf{M}_M = (\Lambda V_M,d)$ is the minimal model of $M$, then 
    \begin{align*}
        V^2_M &= H^2(M) = H^2(X) \oplus \mathrm{span}\{a_{2,S} \, : \, S \in \pi_0(\mathcal{S})\} & d = 0, \\
        V^3_M &= H^3(M) \oplus s \ker \cup \colon \mathsf{S}^2V^2_M \rightarrow H^4(M), & d \colon s \ker \cup \xrightarrow{\cong} \ker \cup.
    \end{align*}
    The model map $\mu \colon \mathsf{M}_M \rightarrow \mathsf{A}(M)$ satisfies 
    \begin{align*}
        & \mu(H^2(X)) \subseteq \mathsf{B}(M) = \Omega(U) \oplus_{\Omega(\partial U)} \bigoplus_{S\in \pi_0(\mathcal{S})} 1 \otimes \Omega(S),\\
        %&\mu(H^2(M)) \subseteq \ker d^2 \qquad \text{ and } \qquad \mu(H^3(M)) \subseteq \ker d^3, \\
        &\mu(a_{2,S}) = (0,a_{2,S} \otimes 1). 
    \end{align*}
    %%% we will see later, that this already determines a lot about the kernel.
\end{lemma}
\vspace{6pt}
We now describe the construction of the fibre bundles that will form the linear independent elements in $\pi_2(B\Diff(M))$.
For each $S \in \pi_0(\mathcal{S})$, we can resolve the `singularity' $S$ in a twisted fashion using the non-trivial bundle $\mathcal{DO}(k)$, which yields a fibre bundle $M \rightarrow E_S \rightarrow S^2$.
Formally, it arises out of the orbifold $X$ as the following pushout:
\begin{equation}\label{eq: Definition Twisted Bundle Construction}
     \xymatrix{ \partial \Tub(\mathcal{S}) \times S^2 \ar[rr] \ar[d] &&  \Bigl(\bigsqcup_{S \neq S' \in \pi_0(\mathcal{S})} S' \times D\mathcal{O}(k_{S'}) \times S^2 \Bigr) \sqcup \mathcal{DO}(k_S)  \ar[d] \\
    U \times S^2 \ar[rr] && E_S. }
\end{equation}
By the pushout property, the projection to $S^2$ induces a projection map $E_S \rightarrow S^2$.
Since $\partial\mathcal{DO}(k) = S^2 \times S^3/\Z_{k}$, the projection map turns $E_S$ into a fibre bundle with fibre $M$. 
The trivialisation of the boundary extends by construction to the `interior' $U \times S^2 \subset E_S$. 
Thus, the classifying map $f_{E_S}$ of $E_S$ is given by the composition
\begin{equation*}
    \xymatrix{ f_{E_S}\colon S^2 \ar[r] & BS^1 \ar[r]^-{B\twist} & B\Diff_\partial(D\mathcal{O}(k_S)) \ar[r]^-{B\extension_S} \ar[d] & B\Diff(M) \ar[d] \\
    && B\hAut_\partial(D\mathcal{O}(k_S)) \ar[r]^-{B\extension} & B\hAut(M),   }
\end{equation*}
where $\extension_S(\varphi)$ is the extension of $\id_S \times \varphi$ by the identity from $S \times D\mathcal{O}(k) \subseteq M$ to the rest of $M$. 

To understand which element the maps $f_{E_S}$ represents in $\pi_2(B\hAut(M))$, or equivalently, which elements their loop maps $\Omega f_{E_S}$ represent in $\pi_1( \Omega B\hAut(M), \id) \cong \pi_1(\hAut(M),\id)$, we consider the following homotopy-commutative diagram
\begin{equation*}
    \xymatrix{ S^1 \ar[rr]^-{\twist} \ar[d]_{\mathrm{Ad}(\id_{S^2})} && \hAut_\partial( D\mathcal{O}(k_S) ) \ar[rr]^-{\extension_S} \ar[d]^\simeq && \hAut(M) \ar[d]^\simeq \\
    \Omega S^2 \ar[rr]^-{\Omega B\twist} && \Omega B\hAut_\partial( D\mathcal{O}(k_S) ) \ar[rr]^-{\Omega f_{E_S}}  && \Omega B\hAut(M) }
\end{equation*}
and remember that $\mathrm{Ad}(\id_{S^2}) \colon \pi_1(S^1) \rightarrow \pi_1(\Omega S^2) = \pi_2(S^2)$ induces an isomorphism.
It is therefore enough to understand the pointed homotopy class of the upper composition.

With help the pushout model $\mathsf{A}(M)$ defined in (\ref{eq: pushout model}) and $H(S^1) \otimes \mathsf{A}(M)$ for the domain $S^1 \times M$,  it is fairly easy to write down a real model for the adjoint map $\extension_T \circ \twist \colon S^1 \times M \rightarrow M$.
Indeed, since it is the extension the map $\twist \times \id_T \colon S^1 \times D\mathcal{O}(k_T) \times T \rightarrow D\mathcal{O}(k_T) \times T$ by the projection to the second factor, a real model for $\extension_T(\twist)$ on the pushout model for $M$ is given by
\begin{equation}\label{eq: pushout model twist extension}
    \xymatrix{\mathsf{A}(M) \ar@{=}[rr] \ar[d] && \Omega(U) \oplus_{\Omega(\partial U)} \bigoplus_{S \in \pi_0(\mathcal{S})} \mathsf{M}_{S^2} \otimes \Omega(S)  \ar[d]^{1 \otimes \id_{\Omega(U)} \oplus \bigoplus_{S} 1 \otimes (\id + \delta_{ST} \theta^1_{\twist})\otimes \id  }  \\
    H(S^1) \otimes \mathsf{A}(M) \ar@{=}[rr] && H(S^1) \otimes \Omega(U) \oplus_{H(S^1) \otimes \Omega(\partial U)} \bigoplus_{S \in \pi_0(\mathcal{S})} H(S^1) \otimes \mathsf{M}_{S^2} \otimes \Omega(S)}
\end{equation}
from which we can easily read-off that the corresponding derivation is 
\begin{equation}\label{eq: twist derivation expression}
    \theta^1_{\extension_T(\twist)} = \xi \cdot a_{2,T} \otimes b_{3,T}^\vee,
\end{equation}
where $\xi$ is the real number in (\ref{cor: derivation of the twist}).
%Composition with the (truncated) minimal model $\mathrm{P}_3\mathsf{M}_{M} = \Lambda V^{\leq 3}_M\xrightarrow{\mu} \mathsf{A}(M) \rightarrow H(S^1)\otimes \mathsf{A}(M)$ yields the element $\theta^1_{\extension_S(\twist)} \circ \mu \in \Der(\mathrm{P}_3 \mathsf{M}_M,\mathsf{A}(M))$.
Observe that $\mathsf{B}(M) = \Omega(U) \oplus_{\Omega(\partial U)} \bigoplus_{S\in \pi_0(\mathcal{S})} 1 \otimes \Omega(S)$ is a differential graded subalgebra, and that $\theta^1_{\extension_T(\twist)}$ vanishes on this subalgebra for every $T \in \pi_0(\mathcal{S})$.
Observe further that each $\theta^1_{\extension_T(\twist)}$ vanished on $\ker d \colon \mathsf{A}^3(M) \rightarrow \mathsf{A}^4(M)$.
We compose the derivations with the model map from the 3-truncation of its minimal model $\mu \colon \mathrm{P}_3\mathsf{M}_M \rightarrow \mathsf{M}_M \rightarrow \mathsf{A}(M)$ to make the calculations for manageable, see Example \ref{emxpl: postnikov decomposition of Sullivan algebras} for the definition of $\mathrm{P}_3\mathsf{M}_M$.

Recall from the introduction that $\mathsf{N} \subseteq \pi_0(\mathcal{S})$ is the nice subset of path components with a homological partner, that is $S \in \mathsf{N}$ if and only if there is a path components $S' \in \pi_0(\mathcal{S})$ different from $S$ such that their homology classes generate the same subvector space $\R\cdot [S] = \R \cdot [S'] \subseteq H_{n-4}(X;\R)$.
\begin{prop}\label{prop: vs of twist derivation}
    The set of homology classes of derivations
    \begin{equation*}
        \{ [\theta^1_{\extension_S(\mathsf{tw})} \circ \mu] \, : \, S \in \mathsf{N} \} \subseteq H_1\bigl( \mathrm{Der}^\mu(\mathrm{P}_3\mathsf{M}_M,\mathsf{A}(M)) ,\delta \bigr)
    \end{equation*}
    spans an $|\mathsf{N}|$-dimensional subvector space in $H_1\bigl( \mathrm{Der}^\mu(\mathrm{P}_3\mathsf{M}_M,\mathsf{A}(M)) ,\delta \bigr)$.
\end{prop}
\begin{proof}
     We will show that if there is a set of coefficients $\{\lambda_S\,:\, S \in \mathsf{N} \}$ such that $\sum \lambda_S \theta_{\extension_S(\twist)}^1 \circ \mu$ lies in the image of $\delta_2$, then all coefficients must vanish.
     This implies, in particular, that the set $\{[\theta_{\extension_S(\twist)}^1 \circ \mu] \, : \, S \in \mathsf{N}\}$ is a linear independent system in $H_1(\Der^\mu(\mathrm{P}_3\mathsf{M}_M,\mathsf{A}(M)),\delta)$. 
    
    \vspace{6pt}

    From (\ref{eq: twist derivation expression}) we observe that the image of each linear combination of $\{\theta^1_{\extension_S(\twist)} \, : \, S \in \pi_0(\mathcal{S})\}$ lies in the ideal $\mathcal{I}$ generated by all $a_{2,S}$, more precisely,
    \begin{equation*}
        \mathcal{I} = \Omega(U) \oplus_{\Omega(\partial U)} \bigoplus_{S \in \pi_0(\mathcal{S})} a_{2,S}\mathsf{M}_{S^2} \otimes \Omega(S).
    \end{equation*}
    Clearly, this ideal intersects the subalgebra $\mathsf{B}(M)$ in a trivial fashion.

    Since derivation from a free algebra are uniquely determined by the image of the underlying vector space, we have 
    \begin{align*}
        \begin{split}
            \Der_2(\mathrm{P}_3\mathsf{M}_M,\mathsf{A}(M)) &\cong \mathrm{Hom}(H^2(X),\mathsf{B}^0(M)) \oplus \mathrm{Hom}(\mathrm{span}\{a_{2,S} \, : \, S \in \pi_0(\mathcal{S}\},\mathsf{B}^0(M)) \oplus \\
            &\quad \ \, \mathrm{Hom}(H^3(M),\mathsf{B}^1(M)) \oplus \mathrm{Hom}(s\ker \cup, \mathsf{B}^1(M))
        \end{split}
    \end{align*}
    because $\mathsf{A}^1(M) = \mathsf{B}^1(M)$ as $\mathsf{M}_{S^2}^1 = \{0\}$, and $\mathsf{A}^0(M) = \mathsf{B}^0(M)$ because $\mathsf{M}^0_{S^2} = \R\cdot 1$.

    Since $\mathsf{A}^0(M)$ is a subvector space of $\Omega^0(U) \oplus \R^{\pi_0(\mathcal{S})}$, its elements can be represented by tuples $(f,v)$ consisting of functions $f\in \Omega^0(U)$ and $v \in \R^{\pi_0(\mathcal{S})} = \mathrm{Map}(\pi_0(\mathcal{S}),\R))$ satisfying the boundary condition $f|_S = v(S)$.
    Hence, if we denote by $a_{2,S}^\vee$ denote the functional that contains $H^2(X)$ in its kernel, and that sends $a_{2,S}$ to $1$ and all other $a_{2,S'}$ to $0$, then $\mathrm{Hom}(\mathrm{span}\{a_{2,S}\},\mathsf{A}^0(M))$ is spanned by elements of the form $(f,v)a_{2,S}^\vee$.

    Now assume that $\theta^1 \circ \mu \in \mathrm{im} (\delta_2)$, so that we can find $F_1 \in \mathrm{Hom}(H^2(X),\mathsf{B}^0(M))$, $F_2 \in \mathrm{Hom}(H^3(M),\mathsf{B}^1(M))$, and $F_3 \in \mathrm{Hom}(s \ker \cup , \mathsf{B}^1(M))$ as well as a finite set $\{(f_T,v_T)a_{2,T}^\vee\}$ such that
    \begin{align}\label{eq: theta in image 1}
        \theta^1\circ \mu = \sum_{S \in \mathsf{N}} \lambda_S \theta^1_{\extension_S(\twist)} \circ \mu &= \sum_{j=1}^3 \delta_2(F_j) + \sum_{T \in \pi_0(\mathcal{S})} \delta_2\bigl(f_T,v_T)a_{2,T}^\vee\bigr) \\
        &= \sum_{j=1}^3 \delta_2(F_j) + \sum_{T \in \pi_0(\mathcal{S})} (df_T,0)a_{2,T}^\vee - (f_T,v_T)a_{2,T}^\vee \circ d\bigr).\nonumber
    \end{align}
    We make the following observations:
    \begin{itemize}
        \item[(i)] $\delta_2(F_j) = 0$ on $\mathrm{span}\{ a_{2,S} \, : \, S \in \pi_0(\mathcal{S}) \}$ because these elements are closed and, by definition, $F_j$ vanishes on this subset.
        \item[(ii)] The derivation $\theta^1 \circ \mu$ is uniquely determined by its restriction to $\mathrm{P}_3\mathsf{M}_M^3 = \mathsf{M}_M^3$, the vector space of elements of degree $3$, because it vanishes on $\mathrm{P}_3\mathsf{M}^2_M = \mathsf{M}_M^2$.
        \item[(iii)] The restrictions $\delta_2(F_2)|_{\mathrm{P}_3\mathsf{M}^3_M} = dF_2|_{\mathrm{P}_3\mathsf{M}^3_M}$ and $\delta_2(F_3)|_{\mathrm{P}_3\mathsf{M}^3_M} = dF_3|_{\mathrm{P}_3\mathsf{M}^3_M}$ takes values in $\mathsf{B}^2(M)$.
        \item[(iv)] Under the decomposition
        \begin{equation*}
            \begin{split}
                \mathsf{S}^2(H^2(M)) = \mathsf{S}^2(H^2(X)) &\oplus H^2(X) \otimes \mathrm{span}\{ a_{2,S} \, : \, S \in \pi_0(\mathcal
            S) \} \\
            &\oplus \mathsf{S}^2(\mathrm{span}\{ a_{2,S} \, : \, S \in \pi_0(\mathcal
            S) \})
            \end{split}
        \end{equation*}
        the derivation $F_1$ vanishes on the third summand, while the derivation $\theta^1 \circ \mu$ vanishes on the second and third summand. 
    \end{itemize}

    From these observations, we draw the following conclusions:
    By plugging $a_{2,T}$ into (\ref{eq: theta in image 1}), we derive together with observation (i) that $0 = 0 + (df_T,0)$, which implies that $f_T$ is closed and hence a constant function. Furthermore, $v_T = f_T|_{\partial U}$ is a constant function, too, with value $v_T$.

    Furthermore, since $\theta \circ \mu$ takes values in the ideal $\mathcal{I}$, which intersects the subalgebra $\mathsf{B}(M)$ in a trivial fashion, we conclude together with observation (ii), (iii) that 
    \begin{equation}\label{eq: consequences for delta(F)}
        \delta_2(F_2) = \delta_2(F_3) = 0, \quad \text{ and } \quad \delta_2(F_1)|_{\mathrm{P}_3\mathsf{M}_M^3} = -F_1 \circ d |_{\mathrm{P}_3\mathsf{M}_M^3}.
    \end{equation}
    Thus, equation (\ref{eq: theta in image 1}) simplifies to 
    \begin{equation*}
        \theta^1 \circ \mu = \sum_{S \in \pi_0(\mathcal{S})} \lambda_S \theta^1_{\extension_S(\twist)} \circ \mu = -\sum_{T \in \pi_0(\mathcal{S})}  v_T(1,1)a_{2,T}^\vee \circ d + F_1 \circ d|_{\mathrm{P}_3\mathsf{M}^3_M}.
    \end{equation*}
    Since $(0,a_{2,S_0}) \cdot (0,a_{2,S_1}) = 0 \in \mathsf{A}(M)$ if $S_0 \neq S_1$, we have $a_{2,S_0}\cdot a_{2,S_1} \in \ker \cup \subseteq \mathrm{P}_3\mathsf{M}_M^4$.
    Denote the corresponding element in $\mathrm{P}_3\mathsf{M}_M^3$ by $s(a_{2,S_0}\cdot a_{2,S_1}) \in s \ker \cup$.
    Since $\mu$ is a chain map, we deduce that $\mu(s(a_{2,S_0}\cdot a_{2,S_1})) \in \ker d \subseteq \mathsf{A}^3(M)$ on which $\theta^1$ vanishes.
    Evaluating the simplified equation at $s(a_{2,S_0}\cdot a_{2,S_1})$, yields 
    \begin{align*}
        0 &= \theta^1(0) = \theta^1\circ \mu(s(a_{2,S_0}\cdot a_{2,S_1})) =   -\sum_{T \in \pi_0(\mathcal{S})}  v_T(1,1)a_2^\vee(a_{2,S_0}\cdot a_{2,S_1}) + F_1(a_{2,S_0}a_{2,S_1})\\
        &= -\sum_{T \in \pi_0(\mathcal{S})} \delta_{TS_0}v_T(1,1)(0,a_{2,S_1}) + \delta_{TS_1} v_T(1,1)(0,a_{2,S_0}) + 0 \\
        &= -v_{S_0}(0,a_{2,S_1}) - v_{S_1}(0,a_{2,S_0}),
    \end{align*}
    where we used observation (iv) to conclude that $F_1(a_{2,S_0}\cdot a_{2,S_1}) =0$.
    Since $\{(0,a_{2,S}) \, : \, S \in \pi_0(\mathcal{S}) \} \subseteq \mathsf{A}^2(M)$ are linearly independent, we deduce $v_S = 0$ for all $S \in \pi_0(\mathcal{S})$.

    \vspace{6pt}

    We are left a collection of real numbers $\lambda_S$ such that
    \begin{equation}\label{eq: hypothetical lin dep}
        \sum_{S \in \mathsf{N}} \lambda_{S} \theta^1_{\extension_{S}(\twist)} \circ \mu = \delta_2(F_1)
    \end{equation}
    for some derivation $F_1 \in \mathrm{Hom}(H^2(X), \mathsf{B}^0(M)) \subseteq \mathrm{Der}_2(\mathrm{P}_3\mathsf{M}_M,\mathsf{A}(M))$.
    We wish to prove that all $\lambda_S = 0$ for all $S \in \mathsf{N}$.

    Let $S_0 \neq S_1 \in \mathsf{N}$ be two path components with $\R\cdot[S_0] = \R \cdot [S_1] \in H_{n-4}(X;\R)$.  
    By Poincar\'e-duality for smooth, closed and oriented orbifolds \cite{Satake1956Orbifolds}, we deduce that the Thom forms of their tubular neighbourhood generate the same cohomology class (up to a non-zero factor) inside $H^4(X,\R)$.
    In the pushout model for $X$, the Thom form of the tubular neighbourhood is modelled by $(0,\vol_{D^4,S})$.
    Since the dga-homomorphism $\Lambda V_{D^4} \rightarrow \mathsf{M}_{S^2} = \mathsf{M}_{D\mathcal{O}(k)}$ induced by the resolution map $\blowdown \colon D\mathcal{O}(k) \rightarrow D^4/\Z_k$ sends maps $\vol_{D^4}$ to $a^2/k$, the dga-homomorphism $\mathsf{A}(X) \rightarrow \mathsf{A}(M)$ induced by the resolution map $\blowdown \colon M \rightarrow X$ sends $(0,\vol_{D^4,S})$ to $(0,a_{2,S}^2/k_S)$.
    We therefore conclude that there exists non-zero coefficients $\nu_{S_0}$ and $\nu_{S_1}$ such that
    \begin{equation*}
        \nu_{S_0}a_0^2 + \nu_{S_1}a_1^2 \in \ker \bigl( \cup \colon \mathsf{S}^2V^2_M \rightarrow H^4(M) \bigr),
    \end{equation*}
    so $s(\nu_{S_0}a_{2,S_0}^2 + \nu_{S_1}a_{2,S_1}^2) \in s\ker \cup \subseteq V^3_M$.

    Obviously, the relation $\sim$ on $\pi_0(\mathcal{S})$ that relates two path components $S$ with $S'$ if and only if their homology classes generate the same vector space in $H_{n-4}(X;\R)$ is an equivalence relation.
    By definition, the set $\mathsf{N}$ partitions into a disjoint union of equivalence classes that contain at least two elements.
    The previous discussion implies that each such equivalence class $S_0/\sim$ generates an $(\# S_0 -1)$-dimensional vector space inside $s\ker \cup$.  

     %$\ker \cup \colon \mathsf{S}^2V^2_M \rightarrow H^4(M)$ spanned by elements of the form $\nu_0a_{2,S_0}^2 + \nu_ja_{2,S_j}^2$ with $\nu_0,\nu_j \neq 0$.

    Since the model map $\mathrm{P}_3\mathsf{M}_M  \xrightarrow{\mu} \mathsf{A}(M)$ is a dga-homomorphism, we necessarily derive from Lemma \ref{lemma: minimal model in low degrees} that
    \begin{equation*}
        \begin{split}
        \mu( s(\nu_{S_0}a_{2,S_0}^2 + \nu_{S_1}a_{2,S_1}^2) ) &= (\varphi_{S_0,S_1} , \nu_{S_0}b_{3,S_0}\otimes 1 + \nu_{S_1}b_{3,S_1} \otimes 1 + 1 \otimes x_{S_0} + 1 \otimes x_{S_1})\\
        & \quad + \sum_{S \in \pi_0(\mathcal{S})} (\varphi_S,a_{2,S} \otimes \psi_S)
        \end{split}
    \end{equation*}
    where $\varphi_S, \varphi_{S_0,S_1} \in \Omega^3(U)$, $x_j \in \Omega^3(S_j)$, and $\psi_S \in \Omega^1(S)$ are closed forms. 

    Plugging $s(\nu_{S_0}a_{2,S_0}^2 + \nu_{S_1}a_{2,S_1}^2)$ into the equation (\ref{eq: hypothetical lin dep}) and using (\ref{eq: twist derivation expression}) together with observation (iv) and (\ref{eq: consequences for delta(F)}) yields
    \begin{align*}
        0 ={}& \delta_2(F_1)(s(\nu_0a_{2,S_0}^2 + \nu_1 a_{2,S_1}^2)) = - F_1(\nu_0a_{2,S_0}^2 + \nu_1 a_{2,S_1}^2) = \sum_{S \in \mathsf{N}} \lambda_S\theta^1_{\extension_S(\twist)} \circ \mu(s(\nu_0a_{2,S_0}^2 + \nu_1 a_{2,S_1}^2)) \\
        ={}&\sum_{S \in \mathsf{N}} \lambda_S\theta^1_{\extension_S(\twist)} ( \nu_{S_0}b_{3,S_0} + \nu_{S_1}b_{3,S_1} ) = \lambda_{S_0}(0,\nu_{S_0}a_{2,S_0}) + \lambda_{S_1}(0,\nu_{S_1}a_{2,S_1}).
    \end{align*}
    By linear independence of $\{(0,a_{2,S}) \, : \, S \in \pi_0(\mathcal{S})\}$ inside $\mathsf{A}^2(M)$, we deduce that $\lambda_{S_0} = \lambda_{S_1} = 0$.  
\end{proof}
%
%The proof of Theorem \ref{Main thm: Orbifold Robustness} is now rather easy.
\begin{proof}[Proof of Theorem \ref{Main thm: Orbifold Robustness}]
    Since the element $[f_{E_S}] \in \mathrm{im}[\pi_2(B\Diff(M)) \rightarrow \pi_2(B\hAut(M))]\otimes \R$ corresponds to $[\extension_{S}(\twist)]$ under the isomorphism $\pi_2(B\hAut(M)) \cong \pi_1(\hAut(M),\id)$, we deduce from Theorem \ref{Main Thm: Homotopy Automorphism via Derivations} (with $X=M$ and $A = \emptyset$ and the real model $\mathsf{A}(M)$ defined in (\ref{eq: pushout model}) that the set of all classes $\{\theta^1_{[\extension_S(\twist)]} \, : S \in \pi_0(\mathcal{S})\}$ generate a subvector space in $H_1(\Der(\mathsf{M}_M,\mathsf{A}(M)),\delta)$.
    The inclusion, $\mathsf{P}_3\mathsf{M}_M \rightarrow \mathsf{M}_M$
    yields a linear map $H_1(\Der(\mathsf{M}_M,\mathsf{A}(M)),\delta) \rightarrow H_1(\Der(\mathrm{P}_3\mathsf{M}_M,\mathsf{A}(M)),\delta)$ that maps the subvector space onto an $\mathsf{N}$-dimensional subvector space $\{\theta^1_{[\extension_S(\twist)]} \circ \mu \, : S \in \pi_0(\mathcal{S})\}$ by Proposition \ref{prop: vs of twist derivation}.
\end{proof}

\section{Applications to $\GTwo$-geometry}\label{Section: G2 Applications}

In this section we apply the previously established topological theory to $\GTwo$-moduli spaces. 
In particular, we will give a proof of Theorem \ref{Main Thm: G2 Manifold fibration detection} and derive Theorem \ref{Main Thm: Detection G2 Moduli Space} from it.
The latter theorem will then be applied to it so a sample of $\GTwo$-manifolds constructed in \cite{Joyce1996CompactG2II}.

We start with the proof of Theorem \ref{Main Thm: G2 Manifold fibration detection}, which is a fairly easy consequence of our topolgical result Theorem \ref{Main thm: Orbifold Robustness}.
\begin{proof}[Proof of Theorem \ref{Main Thm: G2 Manifold fibration detection}]
    Let $T^7/\Gamma$ be a simply connected orbifold and let $\mathcal{S}$ be the union of all path components $S$ of the subspace of singular points that have a tubular neighbourhood $\mathrm{Tub}(S)$ that is diffeomorphic to $T^3 \times D^4/\Z_2$.
    Assume that for $T^7/\Gamma$, there exists a Resolution data ($R$-data) in the sense of Joyce \cite[Definition 14.4.1]{Joyce2000SpecialHolonomy}.
    We resolve all singularities not belonging to $S$ according to the construction presented in \cite[Section 11.4]{Joyce2000SpecialHolonomy} to obtain a orbifold $X$ in which $\mathcal{S}$ is the remaining singular set\footnote{Alternatively, carry out Joyce's construction to obtain $M$ and blow down the resolution of the singularities belonging to $\mathcal{S}$ to obtain  $X$}.
    It comes with a comparison map $\pi \colon X \rightarrow T^7/\Gamma$.
    Because we resolved all singular points not having property (i), the  closed, smooth orbifold $X$ is tailor-made for our purpose by construction.

    We need to prove that the number of all path components with a homological partner does not decrease under this construction. 
    The Poincar\'e-dual of the homology class represented by $T^3 \subset T^3 \times D^4 /\Z_2 \subseteq T^7/\Gamma$ is a four form whose (compact) support is lies in the interior of $T^3 \subset T^3 \times D^4 /\Z_2$.
    By assumption, the tubular neighbourhood $T^3 \times D^4/\Z_2$ does not intersect the tubular neighbourhoods of the remaining singularities, so $\pi \colon \pi^{-1}(T^3 \times D^4/\Z_2) \rightarrow T^3 \times D^4/\Z_2$ is a homeomorphism (that is smooth away from the `cone-tip' $T^3 \times \{0\}$).
    Thus, the Poincar\'e dual pulls back to a differential form whose Poincar\'e dual homology class is represented by $\pi^{-1}(T^3\times \{0\})$.
    It follows that two path components $S \cong T^3$ and $S' \cong T^3$ are homologous in $T^7/\Gamma$ if and only if they are homologous in $X$.
    Theorem \ref{Main thm: Orbifold Robustness} applied to $X$ now yields the claim.
\end{proof}
In the formulation of the next results, which is a refined formulation of Theorem \ref{Main Thm: Detection G2 Moduli Space}, we use notation from the previous section.
\begin{proof}[Proof of Theorem \ref{Main Thm: Detection G2 Moduli Space}]
    The underlying smooth manifold $M$ is obtained from $T^7/\Gamma$ by replacing $S \times D^4/\Z_2$ by $S \times D\mathcal{O}(-2)$ for all $S \in \pi_0(\mathcal{S})$ (and resolving the remaining singularities as explained in \cite[Chapter 11]{Joyce2000SpecialHolonomy}).
    Following Section 6.1 and Section 6.2 in \cite{crowley2025PathComponents}\footnote{In \cite{crowley2025PathComponents}, the spaces $D\mathcal{O}(-2)$ and $\mathcal{DO}(-2)$ are denoted by $EH_{\leq 1}$ and $\mathcal{EH}_{\leq 1}$}, the bundle
    \begin{equation*}
        E_{M,\{S\}} = \bigl(M \setminus S \times D\mathcal{O}(-2) \bigr) \cup_{ S \times \RP^3 \times S^2  } \cup S \times \mathcal{DO}(-2) \rightarrow S^2
    \end{equation*}
    carries a fibre-wise torsion-free $\GTwo$-structure.

    The underlying bundle $E_{M,\{S\}}$ can be alternatively constructed by resolving all singularity components of $T^7/\Gamma$ not belonging to $\mathcal{S}$ to obtain a closed, smooth, simply-connected orbifold $X$ that is tailor-made for our purpose.
    The bundle $E_{S} \rightarrow S^2$ obtained from the construction  (\ref{eq: Definition Twisted Bundle Construction}) obviously agrees with $E_{M,\{S\}}$.
    By Theorem \ref{Main Thm: G2 Manifold fibration detection}, their classifying maps generate an $|\mathsf{N}|$-dimensional vector space inside $\pi_2(B\Diff(M)_0)\otimes \R$.
    Since the fibre bundles carry a fibre-wise torsion free $\GTwo$-structure, this subgroup lifts to an $|\mathsf{N}|$-dimensional subvectorspace in $\pi_2(\hGTwoModuli{M})\otimes \R$, and since the homomorphism  $\pi_2(\hGTwoModuli{M}) \rightarrow \pi_2(\GTwoModuli{M})$ induced by the obvious comparison map is injective, see \cite[Theorem B]{crowley2025PathComponents}, the claim follows.
\end{proof}
We apply this theorem to the examples presented in \cite{Joyce1996CompactG2II}.
\begin{example}\label{exmpl: }
    This is the first class of examples of generalised Kummer constructions discussed in Section 3.1 in \cite{Joyce1996CompactG2II}.
    More precisely, we are discussing Example 3 and Example 4 in detail while providing a sketch for Example 6 in loc. cit. 
    Example 5 in loc. cit. concerns $\GTwo$-manifolds that are not simply connected, so it falls out of our set-up.
    
    The general set-up is the following: For $T^7 = (\R/\Z)^7$ with coordindates $x_1,\dots,x_7$, Joyce considers the group $\Gamma \cong \Z_2^3$ generated by the involutions
    \begin{align*}
        \alpha\bigl( x_1,\dots,x_7\bigr) &= ( -x_1,-x_2,-x_3,-x_4,x_5,x_6,x_7) \\
        \beta\bigl( x_1,\dots,x_7\bigr) &= ( b_1-x_1,b_2-x_2,x_3,x_4,-x_5,-x_6,x_7) \\
        \gamma\bigl(x_1,\dots,x_7\bigr) &= ( c_1-x_1,x_2,c_3-x_3,x_4,-x_5,x_6,-x_7)
    \end{align*}
    with $b_1,b_2,c_1,c_3,c_5 \in \{0,1/2\}$.
    This implies that their fixed-point sets are given by
    \begin{align*}
        \mathrm{Fix}(\alpha) &= \left\{0, \frac{1}{2}\right\}^4 \times T^3(567), \\
        \mathrm{Fix}(\beta) &= \left\{ \left( \frac{b_1+\varepsilon_1}{2}, \frac{b_2+\varepsilon_2}{2}, \frac{\varepsilon_5}{2}, \frac{\varepsilon_6}{2}  \right) \, : \, \varepsilon_j \in \{0,1\}\right\} \times T^3(347), \\
        \mathrm{Fix}(\gamma) &= \left\{ \left( \frac{c_1+\varepsilon_1}{2}, \frac{c_3+\varepsilon_3}{2}, \frac{c_5 + \varepsilon_5}{2}, \frac{\varepsilon_7}{2}  \right) \, : \, \varepsilon_j \in \{0,1\}\right\} \times T^3(246).
    \end{align*}
    From this description it is clear any two canonical identifications of $T^3 \cong (y_1,y_2,y_3,y_4) \times T^3(567)$ yield homotopic maps $T^3 \rightarrow T^7$.
    Thus, each two path components in $\mathrm{Fix}(\alpha)$ generate the same homology class in $H_3(T^7;\R)$ and so their images agree in $H_3(T^7/\Gamma;\R)$.
    The same holds true for the other two group elements.

    In the examples Joyce considers, the other group elements of $\Gamma$ do not have fixed-points, so the singular set of $T^7/\Gamma$ is the image of the fixed-point sets described above under the quotient map $T^7 \rightarrow T^7/\Gamma$.
    It remains to read off the set $\mathcal{S}$ of path components that have a tubular neighbourhood of the form $T^3 \times D^4/\Z_2$ and the nice subset $\mathsf{N}$ from Joyce's article and derive the conclusion from Theorem \ref{Main Thm: Detection G2 Moduli Space}.
    \begin{itemize}
        \item[(i)] In Example 3 in \cite{Joyce1996CompactG2II}, the set $\mathcal{S}$ consists of the image of $\mathrm{Fix}(\alpha)$, $\mathrm{Fix}(\beta)$, and $\mathrm{Fix}(\gamma)$. 
        From the discussion above, we conclude $\mathsf{N} = \pi_0(\mathcal{S})$, which is a set of $12$-elements.
        Theorem \ref{Main Thm: Detection G2 Moduli Space} now implies that $\Z^{12} \subseteq \pi_2( \GTwoModuli{M} )$.
        \item[(ii)] In Example 4 in \cite{Joyce1996CompactG2II}, the set $\mathcal{S}$ is generated by the image of $\mathrm{Fix}(\alpha)$ and $\mathrm{Fix}(\beta)$, and the discussion above shows that $\mathsf{N} = \pi_0(\mathcal{S})$, which is a set of $8$ elements.
         Theorem \ref{Main Thm: Detection G2 Moduli Space} now implies that $\Z^{8} \subseteq \pi_2( \GTwoModuli{M} )$.
    \end{itemize}
    Example 6 in \cite{Joyce1996CompactG2II} uses an additional involution 
    \begin{equation*}
        \delta(x_1,\dots,x_7) = ( 1/2 + x_1, x_2, 1/2 + x_3, 1/2 + x_4, 1/2+ x_5, x_6, x_7)
    \end{equation*}
    and chooses $(b_1,b_2,c_1,c_3,c_5) = (1/2,0,1/2,0,1/2)$.
    The only elements in the group $\langle \alpha, \beta,\gamma,\delta\rangle \cong \Z_2^4$ that have fixed-points are $\alpha,\beta,\gamma,$ and $\alpha\beta\delta$ and, according to Joyce, only $\mathrm{Fix}(\gamma)$ and $\mathrm{Fix}(\alpha\beta\delta)$ contribute each to two path components in $T^7/\Gamma$ which lie in $\mathcal{S}$.
    The fixed-point set of $\alpha\beta\delta$ is given by
    \begin{equation*}
        \mathrm{Fix}(\alpha\beta\delta) = \left\{ \left( \frac{1/2+\varepsilon_3}{2}, \frac{1/2+\varepsilon_4}{2}, \frac{1/2 + \varepsilon_5}{2}, \frac{\varepsilon_6}{2}  \right) \, : \, \varepsilon_j \in \{0,1\}\right\} \times T^3(127)
    \end{equation*}
    and hence we have $\pi_0(\mathcal{S}) = \mathsf{N}$ with $|\mathsf{N}|=4$.
\end{example}
\begin{example}
    We now discuss the Examples 7-11 in Section 3.2 of \cite{Joyce1996CompactG2II} except Example 10, which is not simply-connected.
    The general set-up is the following: Let $\R^7 = \C^3 \times\R$ with standard complex coordinates $z_1,z_2,z_3$ and real coordinate $x$.
    Let $\Lambda \cong \Z^6$ be a lattice and set $T^7 = (\C^3 \times \R)/(\Lambda \times \R)$.
    For two complex numbers $u,v$ with $u^a = v^a =1$ for some natural number $a$, Joyce considers the isometries $\alpha$, $\beta$ defined by
    \begin{align*}
        \alpha(z_1,z_2,z_3,x) &= (uz_1,vz_2,\overline{uv}z_3,x + a^{-1}), \\
        \beta(z_1,z_2,z_3,x) &= (-\bar{z}_1, - \bar{z}_2, - \bar{z}_3, -x).
    \end{align*}
    Under the assumption that the two isometries $\alpha$ and $\beta$ preserve $\Lambda \times \R$ they descend to an action of the dihedral group $\Gamma = \langle \alpha,\beta\rangle \cong D_{2a}$ on $T^7$ that preserves the standard $\GTwo$-structure.
    Joyce observed that $\alpha^j$ does not have fixed-points if $\alpha^j \neq 1$ as it is a translation.
    Furthermore, if $a$ is odd, then $\mathrm{Fix}(\beta\alpha^j)$ and $\mathrm{Fix}(\beta)$
    generated the same singularities and if $a$ is even, then the singularity components in $T^7/D_{2a}$ partition into the ones generated by $\mathrm{Fix}(\beta)$ and $\mathrm{Fix}(\beta\alpha)$.

    Independent of the parity of $a$, the tubular neighbourhoods of all fixed-points are of the form $T^3\times D^4/\Z_2$, see \cite[p.356]{Joyce1996CompactG2II}, which is precisely what we want.
    
    \noindent\underline{Example 7:}
    Here, $u=v=\mathrm{e}^{2\pi\iu/3}$, $a=3$ the lattice $\Lambda = \Z^3 \oplus \mathrm{e}^{2\pi\iu/3}\Z^3$.
    We have $\mathrm{Fix}(\beta) = \R^3/\Z^3 \times \{0\}^3 \times \{0,1/2\} $.
    Clearly, the two tori generate the same homology class in $T^7/D_{2a}$ because the obvious two embeddings are homotopic.
    Thus $\pi_0(\mathcal{S}) = \mathsf{N}$ and $|\mathsf{N}| = 2$.
    It follows that $\pi_2(\GTwoModuli{M}) \supseteq \Z^2$.

    \noindent\underline{Example 8:}
    This one is more interesting: $u = v = \mathrm{e}^{\pi\iu}$, $a=6$, and $\Lambda$ as in Example 7.
    Now $\mathrm{Fix}(\beta)$ is the same as in Example 7, but now $\alpha^3$ identifies the two different component with each other so that they generate a single component in $T^7/D_{12}$.
    Likewise the two components
    \begin{equation*}
       \mathrm{Fix}(\beta\alpha) = \{ (r_1\mathrm{e}^{\pi\iu/6},r_2\mathrm{e}^{\pi\iu/6}, r_1\iu\mathrm{e}^{\pi\iu/3},y) \, : y \in \{3/12,5/12\} \}    
    \end{equation*}
    contributes a single component in $T^7/D_{12}$.
    However, the two path components generate different homology classes in $H_3(T^7/D_{2a};\R) = H_3(T^7;\R)^{D_{2a}}$  %% the equality is an transfer argument that can be found in Bredon - Compact trafo groups p.119
    because the transfer applied to a single component yield elements that are not co-linear over the real numbers.
    Thus, $\mathsf{N}=\emptyset$ and Theorem \ref{Main Thm: Detection G2 Moduli Space} is inconclusive.

    \noindent\underline{Example 9:} 
    Here, $u=v=\iu$, $a=4$ and $\Lambda = \Z^3 \oplus \iu \Z^3$.
    Decompose $z_j = s_j + \iu t_j$ into real and imaginary part.
    Joyce calculated the fixed-point sets to be
    \begin{equation*}
        \mathrm{Fix}(\beta) = \{(s_1 + \varepsilon_1\iu,s_2 + \varepsilon_2\iu,s_3 + \varepsilon_3\iu, \varepsilon_4) \, : \, \varepsilon_j \in \{0,1/2\}\}  
    \end{equation*}
    while 
    \begin{equation*}
        \mathrm{Fix}(\beta\alpha) = \{(s_1 - s_1\iu,s_2 - s_2\iu, t_3 + \varepsilon_3 +t_3\iu, \varepsilon_4) \, : \, \varepsilon_3 \in \{0,1/2\}, x \in \{3/8,7/8\}\}. 
    \end{equation*}
    Since the involution $\alpha^2$ acts freely on $\mathrm{Fix}(\beta)$, the 16 components give rise to eight components in $T^7/\Gamma$. Moreover, all fixed-point components have isotopic inclusions and therefore generate the same homology class.

    Likewise, the four fixed-point components of $\mathrm{Fix}(\beta\alpha)$ contribute to two components in $T^7/\Gamma$ generating the same homology class.
    It follows that $|\mathsf{N}| =\pi_0(\mathcal{S})| = 10$.
    %% Add Example 13 and 14 for the version that will be submitted for publication.

    \noindent\underline{Example 11:}
    Here $u = \mathrm{e}^{\pi\iu/3}$, $v= \mathrm{e}^{2\pi\iu/3}$, $a=6$, and $\Lambda$ the lattice given by
    \begin{equation*}
        \Lambda = (\Z + \mathrm{e}^{2\pi\iu/3}\Z)\oplus (\Z + \mathrm{e}^{2\pi\iu/3}\Z)\oplus (\Z + \iu\Z).
    \end{equation*}
    Joyce computed the fixed-point sets to be
    \begin{equation*}
        \mathrm{Fix}(\beta) = \left\{ (s_1,s_2,s_3 + \varepsilon_3\iu,\varepsilon_4) \, : \, \varepsilon_j \in \{0,1/2\}\right\}
    \end{equation*}
    while
    \begin{equation*}
        \mathrm{Fix}(\beta\alpha) = \left\{ (\lambda_1\mathrm{e}^{\pi\iu/3}, \lambda_2\mathrm{e}^{\pi\iu/6},\varepsilon_3 + t_3\iu,\varepsilon_4) \, : \, \lambda_j,t_3 \in \R, \varepsilon_3 \in \{0,1/2\},\varepsilon_4\in \{5/12,11/12\}\right\}.
    \end{equation*}
    Thus, $\mathrm{Fix}(\beta)$ and $\mathrm{Fix}(\beta\alpha)$ both contribute two singularity components to $T^7/\Gamma$.
    We conclude $|\mathsf{N}| = |\pi_0(\mathcal{S})| = 4$.

    \noindent\underline{Example 13:}
    Set $p = \mathrm{e}^{2\pi\iu/7}$, $u=p$, and $v = p^2$ so that $a=7$.
    In this way, the element $\alpha$ acts on $\C^3 \oplus \R$ as
    \begin{equation*}
        \alpha(z_1,z_2,z_3,x) = (pz_1,p^2z_2,p^4z_3,x + 1/7),
    \end{equation*}
    so the group $\Gamma = \langle \alpha,\beta\rangle$ preserves the lattice
    \begin{equation*}
        \Lambda \oplus \Z= \mathrm{span}_{\Z}\bigl\{ (p^j,p^{2j},p^{4j}) \, : \, 1  \leq j \leq 6 \bigr\} \oplus \Z \subseteq \C^3 \oplus \R
    \end{equation*}
    and therefore descends to a group action on $T^7 = \C^3/\Lambda \times S^1$.
    Since $a =7$ is odd, the singularities of $T^7/\Gamma$ are generated by the image of the fixed-points of $\beta$.
    An explicit calculation reveals
    \begin{equation}
        \mathrm{Fix}(\beta) = \left\{ \left( \mathrm{Re}(p) + t_1\iu, \mathrm{Re}(p^2) + t_2\iu, \mathrm{Re}(p^4) + t_3\iu, \varepsilon/2 \right) \, : \, t_j \in \R, \varepsilon \in \{0,1\} \right\},
    \end{equation}
    so it consists of two disjoint three-dimensional tori with isotopic embeddings into $T^7$.
    We conclude that $\pi_0(\mathcal{S}) = \mathsf{N}$ with $|\mathsf{N}| = 2$.

    \noindent\underline{Example 14:}
    Although this example is pretty close to Example 7, it falls out of the setup described above. 
    However, Joyce proved nonetheless that the sigular set of $T^7/\Gamma$ is generated by the fixed-points of $\beta$, which are the same as in Example $7$.
    However, since $\Gamma$ is a bigger group that in Example $7$, they generate only two connected components in $T^7/\Gamma$ with a tubular neighbourhood of the form $T^3 \times D^4/\Z_2$.
    The embeddings are isotopic, so that $\pi_0(\mathcal{S}) = \mathsf{N}$ with $|\mathsf{N}| = 2$.
\end{example}

\printbibliography

\end{document}